\documentclass[journal,12pt,onecolumn,draftclsnofoot, compsoc]{IEEEtran}

\ifCLASSOPTIONcompsoc
\usepackage[nocompress]{cite}
\else
\usepackage{cite}
\fi
\usepackage{graphicx}
\usepackage{amsmath}
\usepackage{amssymb}
\usepackage{amsfonts, mathrsfs}
\usepackage{tikz}
\usepackage{fancyvrb}
\usepackage{enumerate}
\usepackage{todonotes}
\usepackage{verbatim}
\usepackage{amsthm}
\usepackage{pgfplots}
\usepackage{subcaption}

\usepackage{algorithmic}

\usepackage{mdwmath}
\usepackage{mdwtab}

\usepackage{fixltx2e}

\usepackage{stfloats}

\begin{document}
\newcommand\curly[1]{\left\{#1\right\}}
\newcommand\brac[1]{\left[#1\right]}
\newcommand\paren[1]{\left(#1\right)}
\newcommand\abs[1]{\left|#1\right|}
\newcommand\norm[1]{\left\|#1\right\|}
\newcommand\ang[1]{\left\langle#1\right\rangle}
\newcommand{\R}{\mathbb{R}}
\newcommand{\N}{\mathbb{N}}
\newcommand{\E}{\mathbb{E}}
\renewcommand{\P}{\mathbb{P}}
\newcommand{\cX}{\mathcal{X}}
\newcommand{\cY}{\mathcal{Y}}
\newcommand{\cZ}{\mathcal{Z}}
\newcommand{\calG}{\mathcal{G}}
\newcommand{\calD}{\mathcal{D}}
\newcommand{\calV}{\mathcal{V}}
\newcommand{\Q}{\mathbb{Q}}
\newcommand{\calH}{\mathscr{H}}
\newcommand{\calP}{\mathcal{P}}
\newcommand{\scF}{\mathscr{F}}
\newcommand{\scG}{\mathscr{G}}
\newcommand{\calB}{\mathcal{B}}
\newcommand{\calA}{\mathcal{A}}
\renewcommand{\hat}{\widehat}
\renewcommand{\tilde}{\widetilde}
\renewcommand{\epsilon}{\varepsilon}
\newcommand{\var}{\textrm{Var}}
\makeatletter
\newcommand*{\indep}{%
  \mathbin{%
    \mathpalette{\@indep}{}%
  }%
}
\newcommand*{\nindep}{%
  \mathbin{
    \mathpalette{\@indep}{\not}
  }%
}
\newcommand*{\@indep}[2]{%
  \sbox0{$#1\perp\m@th$}
  \sbox2{$#1=$}
  \sbox4{$#1\vcenter{}$}
  \rlap{\copy0}
  \dimen@=\dimexpr\ht2-\ht4-.2pt\relax
  \kern\dimen@
  {#2}%
  \kern\dimen@
  \copy0 
}
\makeatother

\newtheorem{thm}{Theorem}[section]
\newtheorem{lem}[thm]{Lemma}
\newtheorem{prop}[thm]{Proposition}
\newtheorem{cor}{Corollary}
\newtheorem{defn}{Definition}[section]
\newtheorem{conj}{Conjecture}[section]
\newtheorem{exmp}{Example}[section]
\newtheorem{rem}{Remark}

\title{Conditional Mutual Information Estimation for Mixed Discrete and
  Continuous Variables with Nearest Neighbors}

%
%

\author{Octavio C\'esar Mesner,
  Cosma Rohilla Shalizi
\IEEEcompsocitemizethanks{\IEEEcompsocthanksitem O. C. Mesner is a Ph.D.
  candidate in the Department of Statistics and Data Science and the
  Department of Engineering and Public Policy at Carnegie Mellon
  University, Pittsburgh, PA, 15213.
  This paper was presented in part at 2019 JSM and 2019 Conference of
  Ford Fellows.\protect\\
E-mail: omesner@cmu.edu
\IEEEcompsocthanksitem C. R. Shalizi is an associate professor in the
Department of Statistics and Data Science at Carnegie Mellon University, and an External Professor at the Santa Fe Institute.}
}

\IEEEtitleabstractindextext{%
  \begin{abstract}
    Fields like public health, public policy, and social science often want to
    quantify the degree of dependence between variables whose relationships
    take on unknown functional forms.  Typically, in fact, researchers in these
    fields are attempting to evaluate causal theories, and so want to quantify
    dependence after conditioning on other variables that might explain,
    mediate or confound causal relations.  One reason conditional mutual
    information is not more widely used for these tasks is the lack of
    estimators which can handle combinations of continuous and discrete random
    variables, common in applications.  This paper develops a new method for
    estimating mutual and conditional mutual information for data samples
    containing a mix of discrete and continuous variables.  We prove that this
    estimator is consistent and show, via simulation, that it is more accurate
    than similar estimators.
\end{abstract}

}

\maketitle

\IEEEdisplaynontitleabstractindextext

%
\IEEEpeerreviewmaketitle

\ifCLASSOPTIONcompsoc
\IEEEraisesectionheading{\section{Introduction}\label{sec:introduction}}
\else
\section{Introduction}
\label{sec:introduction}
\fi

\IEEEPARstart{E}{stimating} the dependence between random variables or vectors
from data when underlying the distribution is unknown is central to statistics
and machine learning.  In most scientific applications, it is necessary to
determine if dependence is mediated through other variables.  Mutual
information (MI) and conditional mutual information (CMI) are attractive for
this purpose because they characterize marginal and conditional
independence (they are equal to zero if and only if the variables or vectors in
question are marginally or conditionally independent), and they adhere to the
data processing inequality (transformations never increase information content)
\cite{Dembo-Cover-Thomas-inequalities}.

While there has been limited use of information theoretic statistics in
specific research areas such as gene regulatory networks \cite{liang2008gene,
  hartemink2005reverse, zhang2011inferring}, this has tended to be the
exception rather than the norm.  Typically, it is more common to use
generalized linear regression despite its inability to capture nonlinear
relationships \cite{numata2008measuring}.  This may be, in part, because until
recently, empirically estimating mutual information was only possible for
exclusively discrete or exclusively continuous random variables, a severe
limitation for these fields.

In this paper, we briefly review methods leading up to the estimation of MI and
CMI using distribution-free, nearest-neighbors approaches.  We extend the
existing work to develop an estimator for MI and CMI that can handle mixed data
types with improved performance over current methods.  We prove that our
estimator is theoretically consistent and show its performance empirically.

\section{Background}

The MI between two random variables (or vectors) is a measure of dependence
quantifying the amount of ``information'' shared between the random variables.
The CMI between two random variables given a third is a
measure of dependence quantifying the amount of information shared between
random variables given the knowledge of a third random variable or vector.
These concepts were first developed by Shannon \cite{Shannon-1948}; the
standard modern treatment is \cite{Cover-and-Thomas-2nd}.  These concepts are
inherently linked to entropy and sometimes defined in terms of entropy.

\subsection{Measure Theoretic Information}

Traditionally, the information theoretic measures, entropy and differential
entropy have been used separately for discrete and continuous random variables,
respectively; however, they largely share the same
properties \cite{Cover-and-Thomas-2nd}.  Both of these quantities are equivalent to a
Kullack-Leibler divergence, an expected value of a $\log$-transformed
Radon-Nikodym (RN) derivative, $\E\brac{\log\frac{dP}{d\mu}}$.  The primary
distinction between entropy and differential entropy is the choice of reference
measure, $\mu$, using the Lebesgue measure for continuous variables and the
counting measure for discrete measures.
Lemma \ref{domMeasure} gives the conditions for the existence of a
dominating product measure, $\mu$, in the mixed case.

\cite[\S 5.5]{gray2011entropy} defines entropy and information for generalized
probability spaces as the supremum of all finite, discrete representations
(quantizers) of random variables, mirroring the definition of the Lebesgue
integral.  Because our problem is concerned specifically with mixed
discrete-continuous space, we use this explicit definition which is helpful
when calculating theoretical values and assume all measurable spaces
are standard according to \cite[\S 1.4]{gray2011entropy}.

In order to define MI and CMI, it is necessary that the RN derivative of a
joint probability measure with respect to the product of its marginal
probability measures exists.  The next theorem assures this for
nonsingular joint probability measures.
Further, we assume that all conditional probability measures are regular.

\begin{thm}\label{rnwd}
  Let $P_{XYZ}$ be a joint probability measure on the space
  $\cX\times\cY\times\cZ$, where $\cX,\cY,\cZ$ are all metric spaces.
  If for every value of $Z$, $P_{XY|Z}$ is nonsingular (see
  def. \ref{nonsing}), then
  $\frac{dP_{XY|Z}}{d\paren{P_{X|Z}\times P_{Y|Z}}}$ is well-defined.
\end{thm}
\begin{proof}
  For the RN derivative to exist, the Radon-Nikodym theorem requires that
  $P_{XY|Z}$ is absolutely continuous with respect to $P_{X|Z}\times P_{Y|Z}$,
  $P_{XY|Z} \ll P_{X|Z}\times P_{Y|Z}$.  Because we assume that all conditional
  probabilities are regular, we omit the argument associated with $Z$ and
  proceed as probabilities measures of $X$ and $Y$ as appropriate.

  Assume $A\subseteq \cX\times\cY$ such that $(P_{X|Z}\times P_{Y|Z})(A) = 0$.
  Define $A_1 = \curly{x: P_{Y|Z}(A_x) > 0}\times\cY$,
  $A_2 = \cX\times \curly{y: P_{X|Z}(A_y) > 0}$, and
  $A_3 = \curly{(x,y): P_{X|Z}(A_y) = P_{Y|Z}(A_x) = 0}$.
  Notice that $A\subseteq A_1\cup A_2\cup A_3$.

  From Fubini's theorem, we have that
  \begin{align*}
    0 &= (P_{X|Z}\times P_{Y|Z})(A)\\
    &= \int_\cX P_{Y|Z}(A_x)dP_{X|Z}(x).
  \end{align*}
  Using \cite[Lemma 1.3.8]{dembo2019probability}, $f\geq0, \int
  fd\mu=0\Rightarrow \mu\curly{x:f(x)>0}=0$, for the first equality, we must
  have
  \begin{align*}
    0 &= P_{X|Z}\paren{\curly{x: P_{Y|Z}(A_x) > 0}}\\
    &= P_{XY|Z}\paren{\curly{x: P_{Y|Z}(A_x) > 0}\times\cY}\\
    &= P_{XY|Z}\paren{A_1}.
  \end{align*}
  Using the same construction but switching $X$ and $Y$, we also have that $0 =
  P_{XY|Z}\paren{A_2}$.
  $P_{XY|Z}\paren{A_3} = 0$ follows from the definition of non-singular.

  This shows that $P_{XY|Z} \ll P_{X|Z}\times P_{Y|Z}$.  Now, we may apply the
  RN theorem, so there exists a measurable function, $f$ such that for any
  measurable set $A\subseteq \cX\times\cY$,
  \begin{equation}\int_A fd(P_{X|Z}\times P_{Y|Z}) = P_{XY|Z}(A)\end{equation}
  and $f$ is unique almost everywhere $P_{X|Z}\times P_{Y|Z}$.
\end{proof}

\cite[Lemmas 7.16 and 7.17]{gray2011entropy} shows that if a joint measure is
absolutely continuous with respect to any product measure, then it is
absolutely continuous with respect to its product measure.
Theorem~\ref{rnwd} maybe more helpful for data analysis by showing the
sufficient condition for a nonsingular distribution in the mixed
setting in def. \ref{nonsing}.
Loosely, the RN derivative exists if no continuous variable is a
deterministic function of other variables.

\begin{defn}\label{cmidef}
The \textbf{conditional mutual information} of $X$ and $Y$ given $Z$
is
\begin{equation}\label{cmi}
I(X;Y|Z) \equiv \int{ \log{ \paren{\frac{dP_{XY|Z}}{d\paren{P_{X|Z}\times P_{Y|Z}}}}dP_{XYZ}}}
\end{equation}
where $P_{XY|Z}, P_{X|Z},$ and $P_{Y|Z}$ are regular conditional probability
measures and $\frac{dP_{XY|Z}}{d\paren{P_{X|Z}\times P_{Y|Z}}}$ is the
Radon-Nikodym derivative of the joint conditional measure, $P_{XY|Z}$, with
respect to the product of the marginal conditional measures, $P_{X|Z}\times
P_{Y|Z}$.  If $Z$ is constant, then Eq. \ref{cmi} is $I(X;Y)$, the
\textbf{mutual information} of $X$ and $Y$.
\end{defn}

Definition \ref{cmidef} retains the standard properties of CMI.

\begin{cor}
  \begin{enumerate}
  \item $X$ and $Y$ are conditionally independent given $Z$, $X\indep
    Y|Z$, if and only if $I(X:Y|Z) = 0$.
  \item $I(X;Y|Z) = H(X,Y,Z) - H(X,Z) - H(Y,Z) + H(Z)$
  \item If $X\rightarrow Z\rightarrow Y$ is a Markov chain, the data
    processing inequality states that $I(X;Y)\leq I(X:Z)$.
  \end{enumerate}
\end{cor}

\subsection{Nearest-Neighbor Estimators for Continuous Random
  Variables} 
\label{continuousestimation}

Estimation of entropy, mutual information, and conditional information for
discrete random variables can, in principle, be based on straight-forward
``plug-in'' estimates, substituting the empirical distribution in to the
defining formulas, though such estimates can suffer from substantial
finite-sample bias, especially when the number of categories is large
\cite{Victor-information-bias}, and a range of alternative estimators are also
available.

Estimation for continuous random variables is more challenging.
A direct plug-in estimation would first require estimate of densities,
which is a challenging problem in itself.
Dmitriev and Tarasenko first proposed such an estimator for
functionals~\cite{dmitriev1973functionalEst} for scalar random
variables.
Darbellay and Vajda~\cite{darbellay1999estimation}, in contrast,
proposed an estimator mutual information based on frequencies in
rectangular partitions.
Nearest-neighbor methods of estimating
information-theoretic quantities for continuous random variables which evade
the step of directly estimating a density go back over
thirty years, to \cite{kozachenko1987sample}, which proposed an estimator
of the differential entropy.

\subsubsection{Kozachenko and Leonenko estimator of entropy}\label{kl}

Kozachenko and Leonenko (KL) first used nearest neighbors to estimate
differential entropy \cite{kozachenko1987sample}.  Briefly, let
$X\in\cX\subseteq\R^d$ be a random variable and $x_1,\hdots,x_n\sim P_X$ be a
random sample from $X$.  Estimating the entropy of $X$, as
\begin{equation}\label{sampEnt}
  \hat H(X) = -\frac{1}{n}\sum_{i=1}^n\hat{\log f_X(x_i)}
\end{equation}
where $f_X$ is the density of $X$, we focus on $\hat{\log f_X(x_i)}$ for each
$i$ locally.  Define $\rho_{k,i,p}$ as the $\ell_p$-distance from point $x_i$
to its $k$th nearest neighbor, $k\text{NN}_i$, and $B(x_i, \rho_{k,i,p})$ as
the $d$-dimensional, $\ell_p$ ball of radius $\rho_{k,i,p}$ centered at $x_i$.
Consider the probability mass of $B(x_i, \rho_{k,i,p})$, $P_{k,i,p} \equiv
P_X(B(x_i, \rho_{k,i,p}))$.  $P_{k,i,p}$ could be estimated using the
$d$-dimensional volume in $\ell_p$ of $B(x_i,
\rho_{k,i,p})$ \cite{wang2005volumes} as
\begin{equation}\label{volEst}
  P_{k,i,p} \approx f_X(x_i)c_{d,p}\rho_{k,i,p}^d
\end{equation}
  where
  $c_{d,p} = 2^d\Gamma\paren{1+\frac{1}{p}}^d
  \big/\Gamma\paren{1+\frac{d}{p}}$
if $f_X(x_i)$ were known.
Notice that, intuitively, $P_{k,i,p}\approx \frac{k}{n}$.
In fact, using lemma~\ref{distDeriv} and seeing that the integral is
the same as $\E\brac{\log V}$ for $V\sim\text{Beta}(k,n-k)$,
\begin{equation}\label{digammaExp}
  \E\brac{\log P_{k,i,p}} = \psi(k) - \psi(n)
\end{equation}
where $\psi(x) = d\log \Gamma(x)/dx$ is the digamma function, and does
not depend on choice of $p$.
Substituting the estimate for $P_{k,i,p}$ in approximation
(\ref{volEst}) into the expectation in (\ref{digammaExp}), we have the
estimator for $\hat{\log f_X(x_i)}$: 
\begin{equation}\label{knnlogprob}
  \hat{\log f_X(x_i)} = \psi(k) - \psi(n) -\log c_{d,p} -
  d\log\rho_{k,i,p},
\end{equation}
making the KL estimator
\begin{equation}\label{klEst}
  \hat H(X) = -\psi(k) + \psi(n) + \log c_{d,p}
  +\frac{d}{n}\sum_{i=1}^n \log \rho_{k,i,p}.
\end{equation}
\cite{gao2018demystifying} showed that its bias is $\tilde O(n^{-1/d})$ and
variance is $\tilde O(1/n)$ where $\tilde O$ is the limiting behavior up to
polylogarithmic factors in $n$.

\begin{figure}
  \centering
  \begin{tikzpicture}
    \begin{axis}[
      xlabel = $X$,
      ylabel = $Y$,
      xmin=-6, xmax=6,
      ymin=-6, ymax=6,
      xtick={-2,2},
      ytick={-2,2},
      xmajorgrids=true,
      ymajorgrids=true,
      grid style=dashed,
      yticklabels={,,},
      xticklabels={,,},
      ]
      \addplot [only marks] table[col sep=comma] {scatterPts.csv};
      \node[,circle,fill,inner sep=2pt] at (axis cs:0,0){};
      \node[below right] at (axis cs: 0,0) {$i$};
      \draw[<->, thick](axis cs:5.6, -1.9)--(axis cs:5.6,1.9)
      node[left, midway] {$\rho_{k,i}$};
      \draw[<->, thick] (axis cs:-1.9,-5.6)--(axis cs:1.9,-5.6)
      node[above, midway] {$\rho_{k,i}$};
    \end{axis}
  \end{tikzpicture}
  \caption{The scatter plot above shows point $i$ and its $k$NN where
    $k=2$ on the right vertical dashed line.
    Here  $n_{X,i}=9$ and $n_{Y,i}=6$.}\label{scatter}
\end{figure}
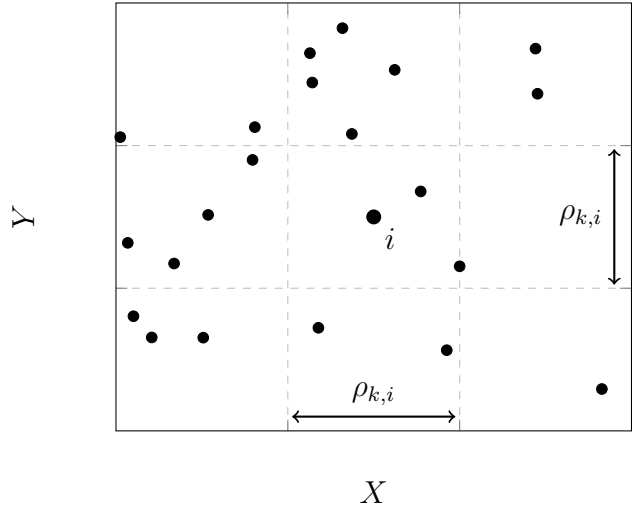

\subsubsection{Kraskov, St{\"o}gbaur, and Grassberger estimator of
  mutual information}\label{ksg} 

Kraskov, St{\"o}gbaur, and Grassberger (KSG) \cite{kraskov2004estimating}
developed an estimator for MI based on $I(X,Y) = H(X) + H(Y) - H(X,Y)$ and a
variation of the KL entropy estimator for continuous random variables or
vectors, $X$ and $Y$ in $\R^{d_X}$ and $\R^{d_Y}$, respectively.  Let
$(x_1,y_1),\hdots, (x_n,y_n)\sim P_{XY}$.  Setting $p = \infty$, define the
$\ell_\infty$-distance from point $(x_i,y_i)$ to its $k$NN as
$\frac{1}{2}\rho_{k,i,\infty}\equiv \frac{1}{2}\rho_{k,i}$,
so that $c_{d,p} = 1$ and $\log c_{d,p} = 0$.
Using this, the local KL estimate for the (negative) joint entropy at
point $i$ is
\begin{equation}
\hat{\log f_{XY}}_i = -\psi(k) + \psi(n) + (d_X + d_Y)\log \rho_{k,i};
\end{equation}
$\hat H(X,Y)$ is computed as in eq. (\ref{sampEnt}).
To calculate $\hat{\log f_{X}}_i$ and $\hat{\log f_{Y}}_i$, the KSG
method deviates slightly from KL by using different values for the $k$ 
hyper-parameter argument for each $i$.
In contrast, $\hat{\log f_{XY}}_i$ used the same value of $k$ for each
$i$ to calculate $\hat H(X,Y)$.
For KL, the $k$ argument can be chosen as any
integer value between 1 and $n-1$ which in turn determines the
$\ell_p$-distance to each point's $k$NN.
Considering each point $i$ separately, KSG works backward for $\hat
H(X)$ and $\hat H(Y)$, by first choosing a distance, $r$, then
counting the number of points that fall within the $\ell_\infty$ ball
of radius $r$ centered at point $i$ within the $X$ (or $Y$) subspace.
It uses this count of points to compute $\hat{\log f_{X}}_i$
(or $\hat{\log f_{Y}}_i$) in place of $k$ hyper-parameter argument and
$r$ in the distance argument.
Specifically, for each $i$, KSG chooses 
$r = \frac{1}{2}\rho_{k,i}$, the 
$\ell_\infty$-distance from point $(x_i,y_i)$ to its $k$NN in
$(\R^{d_X+d_Y},\ell_\infty)$, that was used to calculate
$\hat{\log f_{XY}}_i$.
Call the corresponding count of points $n_{X,i}^*$ in the $X$ subspace and
$n_{Y,i}^*$ in the $Y$ subspace:
\begin{equation}\label{nwisl}
  n_{W,i}^* = \abs{\curly{w_j:\norm{w_i-w_j}_\infty
      <\frac{1}{2}\rho_{k,i}, i\neq j}}
\end{equation}
and
\begin{equation}
  \hat{\log f_W}_i = -\psi(n_{W,i}^*) + \psi(n) + d_W\log \rho_{k,i}
\end{equation}
where $W$ is either $X$ or $Y$.

The KL estimator is accurate because the value of $P_{k,i,p}$, the 
probability in the local neighborhood around $(x_i,y_i)$ extending out to its
$k$NN, is completely determined by $k$ and $n$.  By using the $\ell_\infty$
norm in the KSG estimator, $\frac{1}{2}\rho_{k,i}$ is equal to the
absolute scalar difference between point $(x_i,y_i)$ and $k\text{NN}_i$ at a
coordinate in either $X$ or $Y$.  This way, the entropy estimates for either
$X$ or $Y$ will be accurate in the KL paradigm.  But, the point $k\text{NN}_i$
is not counted in $n_{X,i}^*$ or $n_{Y,i}^*$ because the
definition counts points whose distance from $(x_i,y_i)$ are strictly less than
$\frac{1}{2}\rho_{k,i}$, biasing (\ref{knnlogprob}) toward zero.  Using
$n_{W,i}^*+1$ for $W=X,Y$ corrects this for either $X$ or $Y$ but not
both; that is, either $n_{X,i}+1$ or $n_{Y,i}+1$ will be the
number of points within a distance of exactly $\frac{1}{2}\rho_{k,i}$.
See Fig.\ \ref{scatter}.

Plugging the estimates for $\hat H(X,Y), \hat H(X)$ and $\hat H(Y)$
discussed above into $I(X,Y) = H(X,Y) - H(X) - H(Y)$, we have
\begin{equation}
  \begin{split}
    &\hat I_{\text{KSG}}(X;Y)
    = \psi(k) + \psi(n)\\
    &\quad-\frac{1}{n}\sum_{i=1}^n
  \brac{\psi(n_{X,i}^*+1) +\psi(n_{Y,i}^*+1)}
  \end{split}
\end{equation}
where the $\frac{d}{n}\sum_{i=1}^n\rho_{k,i}$ terms all cancel
from using the same value of $\rho_{k,i}$ for each $i$
and $\log c_{d,p}$ is zero using the $\ell_\infty$ norm and choosing to
set the $k$NN distance to $\frac{1}{2}\rho_{k,i}$.

\cite{kraskov2004estimating} did not offer any proofs on convergence.
Attempting to correct the counting error mentioned, \cite{kraskov2004estimating}
provided another, less-used estimator as well.
Gao \cite{gao2018demystifying} later showed that
the KSG estimator is consistent with a bias of $\tilde
O\paren{n^{-\frac{1}{d_X+d_Y}}}$ and a variance of $\tilde O\paren{1/n}$.

\subsubsection{Frenzel and Pompe estimator of conditional mutual information}\label{fp}

Using a similar technique to estimate conditional mutual information, Frenzel
and Pompe (FP) first, though several other papers as well
\cite{frenzel2007partial, vejmelka2008inferring, tsimpiris2012nearest,
  runge2018conditional, rahimzamani2017cmi} used $I(X;Y|Z) = H(X,Y,Z) - H(X,Z)
- H(Y,Z) + H(Z)$ combined with the KSG technique to cancel out the
$\rho_{k,i}$ term from each of the entropy estimators to estimate CMI as
\begin{equation}\label{fplocal}
  \begin{split}
    \xi_i &= \psi(k)-\psi(n_{XZ,i}^*+1)\\
    &\quad-\psi(n_{YZ,i}^*+1) + \psi(n_{Z,i}^*+1)
  \end{split}
\end{equation}
where $n_{W,i}$ is calculated as in equation \ref{nwisl} with $W = XZ,
YZ, Z$.  The global CMI estimator, $\hat I_{\text{FP}}(X;Y|Z)$, is calculated
by averaging overall $\xi_i$.  While these papers show that this estimator does
well empirically, they do not provide theoretical justification.

\subsection{Estimation for Mixed Variables} \label{mixedestimation}

\subsubsection{Gao, Kannan, Oh, and Viswanth estimator of mutual
  information}\label{gkov}

Gao, Kannan, Oh, and Viswanth (GKOV) \cite{gao2017estimating} expanded on the
KSG technique to develop an MI estimator for mixes of discrete and continuous,
real-valued random variables.  In this setting, unlike the purely continuous
one, there is some probability that multiple independent observations will be
equal.  Depending on the value of $k$, there is a corresponding, nonzero
probability that the $k$NN distance is zero for some points.  While this
impairs the KL entropy estimator due to the $\log \rho$ term, the KSG estimator
only uses the $k$NN distance for counting points with that radius.  Similar to
the insight for the KSG technique, \cite{gao2017estimating} allows $k$ to
change for points whose $k$NN distance is zero:
\begin{equation}\label{ktildeGao}
\tilde k_i^* = \abs{\curly{(x_j,y_j): \norm{(x_i,y_i) - (x_j,y_j)}_\infty =
  0, i\neq j}}.
\end{equation}
To accommodate points whose $k$NN distance is zero, \cite{gao2017estimating}
changes the definition of $n_{W,i}^*$ to include boundary points:
\begin{equation}\label{nwi}
  n_{W,i} = \abs{\curly{w_j:\norm{w_i-w_j}_\infty
      \leq\frac{1}{2}\rho_{k,i}, i\neq j}}
\end{equation}
where $\frac{1}{2}\rho_{k,i}$ remains the $\ell_\infty$
distance from point $(x_i,y_i)$ to $k\text{NN}_i$.
For index $i$, \cite{gao2017estimating} locally estimates MI as
\begin{equation}
  \begin{split}
    \xi_i &= \psi(\tilde k_i^*) + \log(N)\\
    &\quad- \log(n_{X,i}+1) - \log(n_{Y,i}+1).
  \end{split}
\end{equation}
The global MI estimate is the average of the local estimates for each
point:
\begin{equation}
\hat I_{\text{GKOV}}(X;Y) = \frac{1}{n}\sum_{i=1}^n\xi_i ~.
\end{equation}
\cite{gao2017estimating} shows that this estimator is consistent under some
mild assumptions.

Rahimzamani, Asnani, Viswanath, and Kannan (RAVK)
\cite{rahimzamani2018mixedcmi} extend the idea of \cite{gao2017estimating} for
estimating MI for mixed data to a concept the authors define as \textit{graph
  divergence measure}, a generalized Kullback-Leibler (KL) divergence between a
joint probability measure and a factorization of the joint probability measure.
The authors say that this can be thought of as a metric of incompatibility
between the joint probability and the factorization.

Setting the factorization of $P_{XYZ}$ to $P_{X|Z}P_{Y|Z}P_Z$ gives an
equivalent definition of \ref{cmidef} of conditional mutual
information.  Using this factorization, the GKOV estimator for CMI at
index $i$ is 
\begin{equation}\label{ravkCMI}
  \begin{split}
    \xi_i &= \psi(\tilde k_i) - \log(n_{XZ,i}+1)\\
    &\quad - \log(n_{YZ,i}+1) + \log(n_{Z,i}+1).
  \end{split}
\end{equation}
The authors state that $\tilde k_i$ is the number of points within,
$\rho_{k,i}$, the distance to the $k$NN, of observation $i$.
Giving more detail, case III in the proof for
\cite[theorem 2]{rahimzamani2018mixedcmi}
states that $\rho_{k,i} > 0$ implies that $\tilde k_i = k$,
suggesting that $\tilde k_i$ is defined the same as (\ref{ktildeGao}).
Similarly, the proofs suggest that $n_{W,i}$ is defined as (\ref{nwi}).
The global CMI,
$I_{\text{GKOV}}(X;Y|Z)$ is calculated by averaging over all $\xi_i$.  This
paper shows that this estimator is consistent with similar assumptions
to those found in \cite{gao2017estimating}.

\section{Proposed Information Estimators}\label{prop}

The estimator for CMI (and MI) proposed in this paper builds on the ideas in
the previous papers but with critical changes that improve performance.
Start by considering local CMI estimates as in (\ref{fplocal}).
As discussed in \S \ref{ksg},
for index $i$, each negative local entropy estimate (\ref{knnlogprob}),
$\hat{\log f_{XYZ}}_i, \hat{\log f_{XZ}}_i, \hat{\log f_{YZ}}_i$, and
$\hat{\log f_{Z}}_i$, (before terms cancel) is accurate in the KL paradigm
when the distance from $(x_i,y_i,z_i)$ to its $k$NN, $n_{XZ,i}$NN,
$n_{YZ,i}$NN, and $n_{Z,i}$NN for each respective subspace ($XYZ, XZ,
YZ$ or $Z$) is exactly $\frac{1}{2}\rho_{k,i}$.
Moving from exclusively continuous data, where ties occur with
probability zero, to mixed data where ties occur with nonzero
probability, required that $n_{W,i}^*$ from (\ref{nwisl})
include boundary points as in $n_{W,i}$ from (\ref{nwi}).  With this
change, the entropy estimates are frequently accurate using $n_{W,i}$
rather than $n_{W,i}+1$ for $W=XZ, YZ, Z$ for continuous data.

With the $\ell_\infty$ norm, the $k$NN distance value,
$\frac{1}{2}\rho_{k,i}$, is equal to the scalar distance in at least
one coordinate of the random vector $(X,Y,Z)$.
If this coordinate is in $Z$, then the distance term in each local entropy
estimate from (\ref{knnlogprob}) will be exactly $d\log \rho_{k,i}$
for $\hat{\log f_{XYZ}}_i, \hat{\log f_{XZ}}_i, \hat{\log f_{YZ}}_i$, and
$\hat{\log f_{Z}}_i$ (because each contains $Z$).
Again, this is because within each given subspace, $\rho_{k,i} =
\rho_{n_{W,i},i}$, for $W = XZ, YZ, Z$, and thus
\begin{align*}
  \xi_i
  &= -\hat{\log f_{XYZ}}_i+\hat{\log f_{XZ}}_i
   +\hat{\log f_{YZ}}_i - \hat{\log f_{Z}}_i\\
  &= \psi(k)-\psi(n_{XZ,i}) -\psi(n_{YZ,i}) + \psi(n_{Z,i})
\end{align*}
with perfect cancellations.

If the $\ell_\infty$-distance coordinate is in $X$, then $\rho_{k,i} =
\rho_{n_{XZ,i},i}$, so that the corresponding terms in $\hat{\log
  f_{XYZ}}_i$ and $\hat{\log f_{XZ}}_i$ cancel but the other two
distance terms may not.
An analogous argument can be made for $Y$.
If the dimension of $Z$ is greater than $X$ and $Y$, heuristically, one
might expect the $k$NN distance to fall in the larger $Z$ dimension.

Theorem \ref{hiDim} will show the proposed estimator tends to zero as
the dimension of the $Z$ vector increases.
The methods discussed in \S~\ref{mixedestimation} will also converge to zero
as the dimension increases; however, the proposed method is an
improvement, especially for discrete points.
The combined dimension of $(X,Y,Z)$ can affect the value of $\tilde k_i$, on
discrete data, when the $k$NN distance is greater than zero.  Consider the case
where data is comprised of exclusively discrete random variables, that
is; each point in the sample has a positive probability point mass.
As the dimension of $(X,Y,Z)$ grows, probability point masses will
diminish as long as the added variables are not determined given the
previous variable.
Moreover, point masses in higher-dimensional spaces will necessarily be less
than or equal to their corresponding locations in lower-dimensional spaces.  It
is possible that the $k$NN distance for index $i$ is zero, especially if it has
a large probability point mass relative to $n$.  But, if its point mass in the
$XYZ$-space is not sufficiently large to expect more than one point at its
location for the given sample size, $n$, we would expect its $k$NN distance to be
greater than zero.
If the $\tilde k_i = k$, as it would in
eq. (\ref{ktildeGao}) because $\frac{1}{2}\rho_{k,i} > 0$, then
$n_{XZ,i}, n_{YZ,i}$, and $n_{Z,i}$ will be the
total number of points within the distance to the $k$NN including points on the
boundary for the appropriate subspace, $XZ, YZ$ or $Z$.
But, because the data are discrete, it is possible/likely that the
$k$NN is not unique.
This would indicate that there are more than $k$, points at and within
the same radius, $\frac{1}{2}\rho_{k,i}$ in the $XYZ$-space.
Under counting here would bias the local estimate of CMI
(\ref{ravkCMI}) downward because $k$ would be small relative to the
values $n_{XZ,i}, n_{YZ,i}$, and $n_{Z,i}$, in the associated
subspaces.
To fix this, we set $\tilde k_i$ to the number of points that are less
than or equal to the $k$NN distance from point $(x_i,y_i,z_i)$ as 
\begin{equation}\label{newtildek}
  \tilde k_i = \abs{\curly{w_j: \norm{w_i-w_j}\leq
      \frac{1}{2}\rho_{k,i}, i\neq j}}.
\end{equation}
Notice that if the data are all continuous, then $\tilde k_i = k$ with
probability one so this change will only affect discrete points.

Moving to the use of the digamma function, $\psi$, verses the natural
logarithm, $\log$, the methods presented
in \S \ref{continuousestimation} to estimate MI and CMI for continuous data use
the digamma function, $\psi$, and not $\log$.  In contrast, the methods in \S
\ref{mixedestimation} use both when estimating MI and CMI for mixed data.
Though no explicit reason is given for the deviation, it seems innocuous given
that $\abs{\log(w) - \psi(w)}\leq \frac{1}{w}$ for $w>0$, and, possibly
reasonable given that the plug-in estimator of CMI on discrete data is
$\log(\tilde k_i)-\log(n_{XZ,i}) - \log(n_{YZ,i}) +
\log(n_{Z,i})$ similar to (\ref{fplocal}), with the difference being
that it uses $\log$ in place of digamma.  But, the use of digamma was developed
specifically in the context of continuous data with no ties.  For this reason,
we use $\psi$ for continuous data and log for discrete data.
If a variable/coordinate  of $(X,Y,Z)$ is categorical (non-numeric),
our code uses the discrete distance metric for that coordinate in the
random vector, in place of absolute difference: 
the coordinate distance is zero at that coordinate for two observations when
equal and one otherwise.
Several cited lemmas and theorems used in the proofs in
\S~\ref{consist} assume vectors to be in $\R^d$.
Categorical variables do not strictly satisfy this requirement but
transforming categorical variables to dummy indicators (as one does
in regression) yields an isometry between the categorical space with
the discrete metric and $\R^m$ where the variable takes $m+1$ distinct
values with an $\ell_\infty$ metric. 
While it is not necessary to create dummy variables for the code to
work, we can be assured that the proofs are satisfied even when data
include categorical data.

To calculate the proposed local CMI estimate for index $i$, determine $\tilde
k_i$ using (\ref{newtildek}), saving the $k$NN distance,
$\frac{1}{2}\rho_{k,i}$.  Next for $W\in\curly{XZ,YZ,Z}$,
determine $n_{W,i}$ from eq.~(\ref{nwi}) using $\frac{1}{2}\rho_{k,i}$.
For each $i\in\curly{1,\dots,n}$, define
\begin{equation}\label{cmiprop}
  \xi_i = \begin{cases}
    \begin{split}
      &\psi(k)-\psi(n_{XZ,i})\\
      &\quad-\psi(n_{YZ,i}) + \psi(n_{Z,i})
    \end{split}
    \text{when } k=\tilde k_i\\
    \begin{split}
      &\log(\tilde k_i)-\log(n_{XZ,i})\\
      &\quad-\log(n_{YZ,i}) + \log(n_{Z,i})
    \end{split}
    \text{when } k<\tilde k_i\\
  \end{cases}.
\end{equation}
The sample estimate for the proposed CMI estimator is
\begin{equation}\label{propest}
  \hat I_{\text{prop}}(X;Y|Z) =
  \max\curly{\frac{1}{n}\sum_{i=1}^n\xi_i,0}.
\end{equation}
To calculate MI between $X$ and $Y$, we can make $Z$ constant
according to def.~\ref{cmidef} so that $n_{Z,i} = n$.
We define CMI and MI as the positive part of the mean because CMI and
MI are provably non-negative.
This setting can easily be changed in the code with a function argument.
In the simulations shown in \S~\ref{experiments}, we display the
mean itself for greater visibility.

\subsection{Consistency}\label{consist}

The estimator proposed in \S \ref{prop} is consistent for fixed-dimensional
random vectors under mild assumptions.  Theorem \ref{unbiased} shows that the
estimator is asymptotically unbiased and theorem \ref{variance} shows that its
asymptotic variance is zero.

As shorthand notation, we set $f \equiv \frac{dP_{XY|Z}}{d\paren{P_{X|Z}\times
    P_{Y|Z}}}$ and for a random variable $W$ on $\mathcal{W}$ with probability
measure, $P_W$, and $w\in\mathcal{W}$, define
\begin{equation}
  P_W(r) = P_W\paren{\curly{v\in\mathcal{W}:
      \norm{v-w}_\infty\leq r}}.
\end{equation}

\begin{thm}\label{unbiased}
  Let $(x_1,y_1,z_1), \hdots, (x_n,y_n,z_n)$ be an i.i.d. random sample
  from $P_{XYZ}$.
  Assume the following:
  \begin{enumerate}
  \item $k=k_n\rightarrow\infty$ and $\frac{k_n}{n}\rightarrow 0$ as
    $n\rightarrow\infty$. \label{knzero}
  \item For some $C>0$, $f(x,y,z)\leq C$ for all $(x,y,z)\in
    \cX\times\cY\times\cZ$.\label{fBound}
  \item $\curly{(x,y,z)\in\cX\times\cY\times\cZ:P_{XYZ}((x,y,z))>0}$
    is countable and nowhere dense in $\cX\times\cY\times\cZ$
  \end{enumerate}
  then
  \begin{equation}
    \lim_{n\rightarrow\infty}{\E\brac{\hat I_{\text{prop}}(X,Y|Z)}} =
    I(X,Y|Z) ~.
  \end{equation}
\end{thm}
The proof can be found in App.\ \ref{consistencyProof}.

\begin{thm}\label{variance}
  Let $W = \{W_1, \hdots, W_n\}$ be a random samples of size $n$ such that for
  each $i$, $W_i = (X_i, Y_i, Z_i)$, $k\geq 2$, and let $\hat I_n(W) =
  \frac{1}{n}\sum_{i=1}^n \xi_i(W)$ where $\xi_i(W) = \xi_i$ as defined above.
  Then
  \begin{equation}
    \lim_{n\rightarrow\infty}{\var\paren{\hat I_{\text{prop}}(W)}} = 0~.
  \end{equation}
\end{thm}
The proof can be found in App.\ \ref{varProof}.

\begin{cor}\label{conv}
  Let $W_1,\dots,W_n$ be an independent,
  identically distributed sample.
  Then
  \begin{equation}
    \begin{split}
      P&\paren{\abs{\hat I_n(W) - \E\brac{\hat I_n(W)}} > t}\\
      &\leq 2\exp\curly{\frac{-t^2n}{2592k^2\gamma_d^2(\log n)^2}}
    \end{split}
  \end{equation}
  where $\gamma_d$ is a constant that is only dependent on the
  dimension of $W$.
\end{cor}
The proof can be found in App.\ \ref{convProof}.

Despite the estimator's unbiasedness in large samples, it is biased
toward zero on high-dimensional data with a fixed sample size,
suffering from the curse of dimensionality as $k$NN regression does.

\begin{thm}\label{hiDim}
  Assume $X$ and $Y$ have fixed-dimension and that $Z=
  (Z_1,Z_2,\hdots,Z_d)$ is 
  a $d$-dimensional random vector.
  If the entropy rate of $Z$ is nonzero, that is,
  $\lim_{d\rightarrow\infty}\frac{1}{d}H(Z)\neq 0$,
  then
  $\hat I_{\text{prop}}(X,Y|Z) \xrightarrow{P} 0$ (converges in
  probability) as $d\rightarrow\infty$.
\end{thm}
The proof can be found in App.\ \ref{hiDimProof}.

\section{Experiments}\label{experiments}

To evaluate the empirical performance of the proposed estimator, we compared it
to the FP estimator for continuous variables found in
\S~\ref{fp} and to two versions of the RAVK estimator for CMI in
\S~\ref{mixedestimation} on simulated mixed data from various
setting.
Both RAVK1 and RAVK2 are calculated using eq (\ref{ravkCMI}), but
using different values for $\tilde k_i$.
RAVK1 uses eq (\ref{ktildeGao}) and RAVK2 uses eq (\ref{newtildek}).
The FP estimator, as it was designed for exclusively continuous data,
when $\rho_{k,i} = 0$ will compute $0$ for $n_{w,i}^*$ from eq~\ref{nwisl},
so $\psi$ will be undefined.
In the simulations, we used $\max\curly{n_{w,i}^*, 1}$.
For greater visibility, all figures show positive and negative
estimator values (even though CMI is non-negative).
Specifically, all figures show the proposed estimator as
$\frac{1}{n}\sum_{i=1}^n\xi_i$ rather than equation~(\ref{propest}).
All simulation data, methods
code, and visuals were done in Python~3.6.5.  We simulated data from differing
distributions with 100 observations up to 1000 in intervals of 100.  The violin
plots in Fig.\ \ref{sim1}--\ref{sim4} show the distribution of estimates from
100 simulated datasets for each sample size.  The ``$\times$''
markers in each violin plot indicates the mean of all estimates and
the $-$ represent to most extreme values.
For both the proposed and continuous estimator, we used $k=7$ for all
datasets.

The first simulation (Fig.\ \ref{sim1}) was inspired by
\cite[example 4.4.5]{casella2002}.
In this scenario, a mother insect lays eggs at a random rate,
$X\sim$~Exponential$(10)$.
The number of eggs she lays is $Z\sim$~Poisson$(X)$, and
the number of the eggs that survive is $Y\sim$~Binomial$(Z,0.5)$.
In this Markov chain ($X\rightarrow Z\rightarrow Y$), $X$ and $Y$ are
marginally dependent $X\nindep Y$ but independent conditioning on $Z$, $X\indep
Y|Z$ so that $I(X;Y|Z) = 0$.

The second simulation (Fig.\ \ref{sim2}) is from \cite{gao2017estimating}: 
$X\sim$~Discrete Uniform$(0,3)$ and $Y\sim$~Continuous Uniform$(X,X+2)$ with an
additional, independently generated random variable, $Z\sim$~Binomial$(3,
0.5)$.
Here, $X\nindep Y|Z$ and $I(X;Y|Z) = \log 3 - 2\log 2/2$.
This example, a combination of discrete and continuous random
variables is common in many applications.
Here, the discrete variables are numeric but it is also reasonable to
use zero-one distance metric for non-numeric categorical variables.

The third simulation (Fig.\ \ref{sim3}) places probability mass of $0.4$ at
both $(1,1)$, and $(-1,-1)$ and probability mass of $0.1$ at $(1,-1)$ and
$(-1,1)$ with an independently generated $Z\sim$~Poisson$(2)$.  In this case,
$X\nindep Y|Z$ with $I(X;Y|Z) = 2\cdot 0.4\log(0.4/0.5^2) + 2\cdot
0.1\log(0.1/0.5^2)$.
In this example, all variables are discrete.

The fourth simulation is also from \cite{gao2017estimating}.  $X$ and $Y$ are a
mixture distribution where with probability $\frac{1}{2}$, $(X,Y)$ is
multivariate Gaussian with a correlation coefficient of 0.8 and with
probability $\frac{1}{2}$, $(X,Y)$ places probability mass of 0.4 at $(1,1)$,
and $(-1,-1)$ and probability mass of 0.1 at $(1,-1)$ and $(-1,1)$, as in the
third experiment.  $Z$ is an independently generated Binomial$(3,0.2)$ so that
$I(X;Y|Z) = I(X;Y)$.  We separate the domain of the integral into its discrete
and continuous parts; that is, $(1,1)$, $(-1,-1)$, $(1,-1)$ and $(-1,1)$ make
up the discrete part and everywhere else the continuous part.  From here we
calculate MI on each partition by multiplying the distribution by
$\frac{1}{2}$, yielding $I(X;Y|Z) = 0.4\log(2 \cdot 0.4/0.5^2) + 0.1\log(2
\cdot 0.1/0.5^2) + 0.125\log(4/(1 - 0.8^2))$.  Results are in Fig.\ \ref{sim4}.
In HIV research, for example HIV viral load, the amount of
virus in a milliliter of a patient's blood can only be measured to a minimum
threshold.
Below that threshold, depending on the assay used, a patient is said
to be undetectable. 
This is a real-world example of a random variable that is itself a mix
of discrete and continuous, difficult for most regression models.
The this experiment shows that the proposed estimator has no problem
in this scenario.

\begin{figure}[!t]
  \centering
  \includegraphics[page = 1, width = 3.5in]{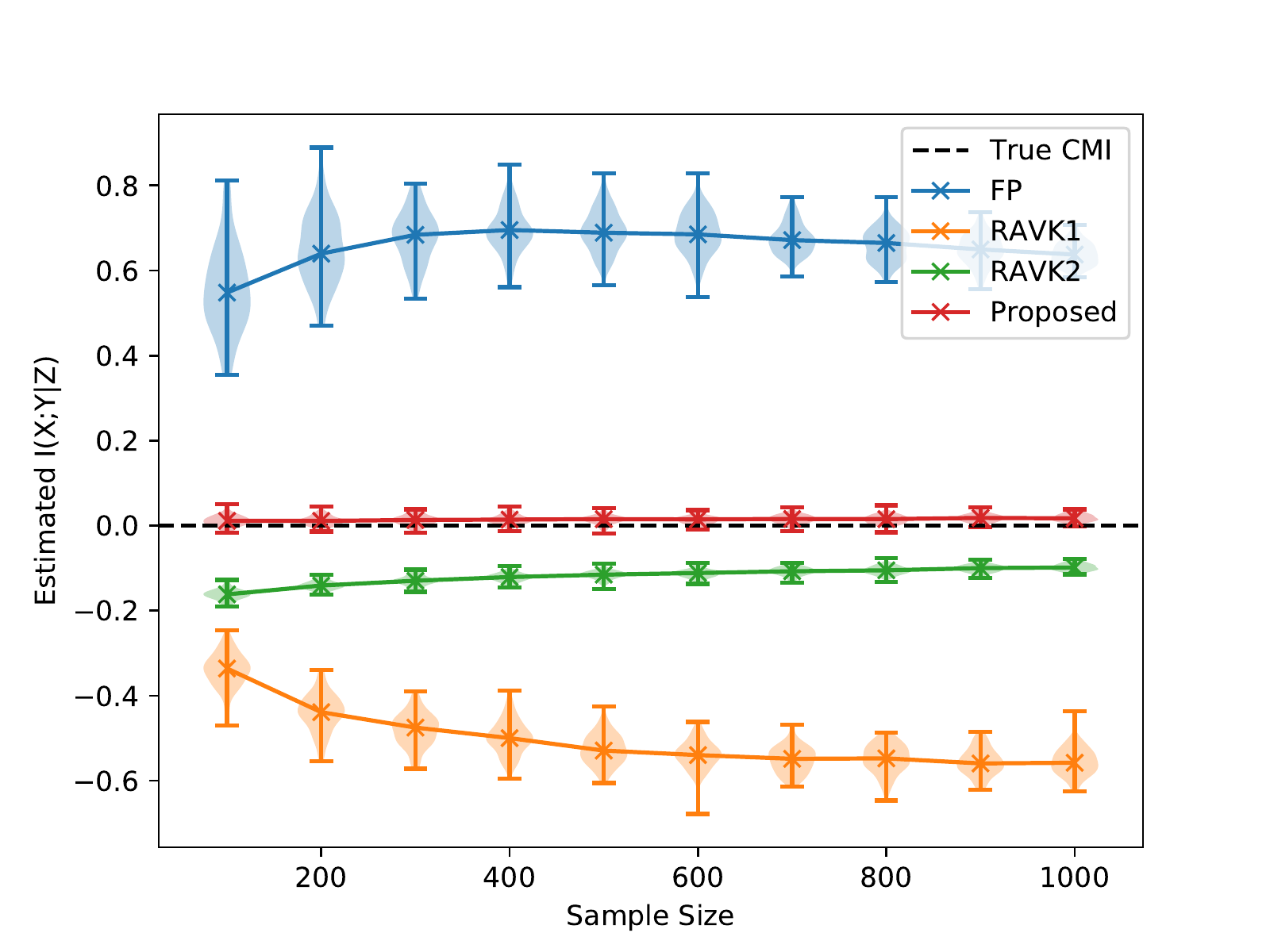}
  \caption{$X\sim$~Exponential$(10)$, $Z\sim$~Poisson$(X)$, and
    $Y\sim$~Binomial$(Z,0.5)$, $I(X;Y|Z) = 0$.}
  \label{sim1}
\end{figure}

\begin{figure}[!t]
  \centering
  \includegraphics[page = 2, width = 3.5in]{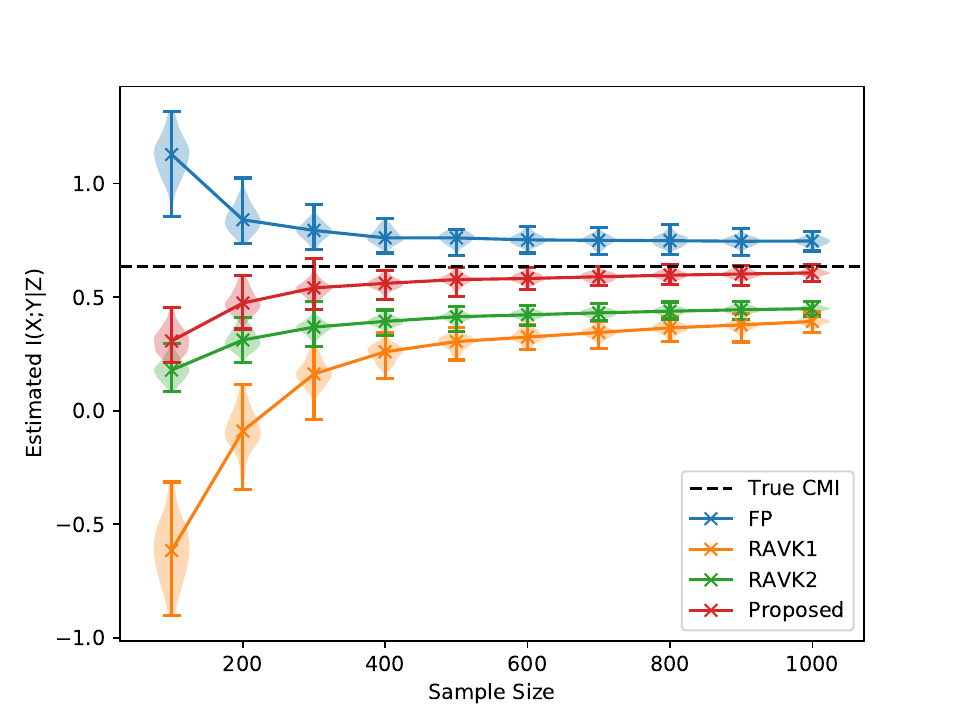}
  \caption{$X\sim$~Discrete Uniform$(0,3)$ and $Y\sim$~Continuous
    Uniform$(X,X+2)$, $Z\sim$~Binomial$(3, 0.5)$, $I(X;Y|Z) = \log 3 -
    2\log 2/2$.}
  \label{sim2}
\end{figure}

\begin{figure}[!t]
  \centering
  \includegraphics[page = 4, width = 3.5in]{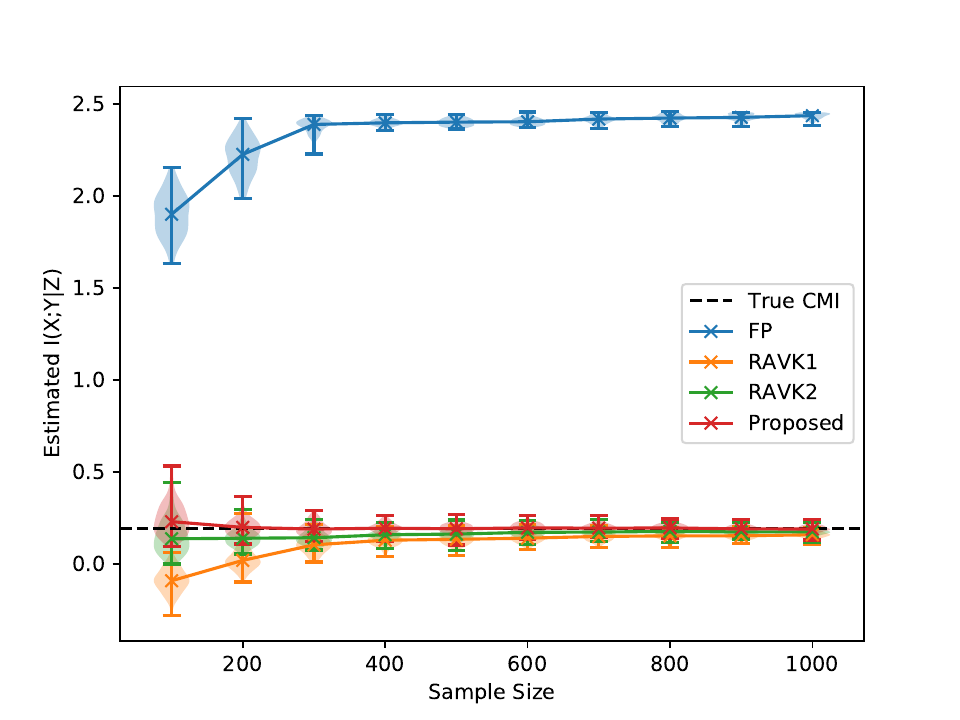}
  \caption{$\P((X,Y) = (1,1)) = \P((X,Y) = (-1,-1)) = 0.4$,
    $\P((X,Y) = (1,-1)) = \P((X,Y) = (1,-1)) = 0.1$,
    $Z\sim$~Poisson$(2)$, $I(X;Y|Z) = 2\cdot
    0.4\log(0.4/0.5^2) + 2\cdot 0.1\log(0.1/0.5^2)$.}
  \label{sim3}
\end{figure}

\begin{figure}[!t]
  \centering
  \includegraphics[page = 3, width = 3.5in]{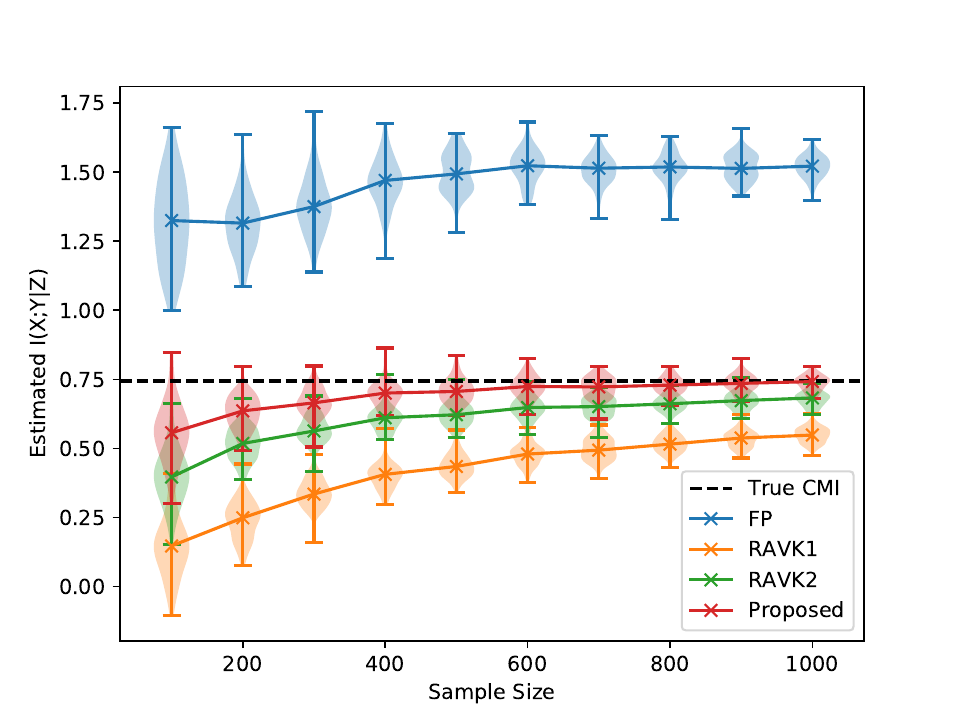}
  \caption{$X$ and $Y$ are a mixture distribution where with probability
    $\frac{1}{2}$, $(X,Y)$ is multivariate Gaussian with a correlation
    coefficient of 0.8 and with probability $\frac{1}{2}$, $(X,Y)$ places
    probability mass of 0.4 at $(1,1)$, $(-1,-1)$ and probability mass of
    0.1 at $(1,-1)$ and $(-1,1)$, as in the third experiment.
    $Z$ is an independently generated Binomial$(3,0.2)$.
    $I(X;Y|Z) = 0.4\log(2 \cdot 0.4/0.5^2) +
    0.1\log(2 \cdot 0.1/0.5^2) + 0.125\log(4/(1 - 0.8^2))$}
  \label{sim4}
\end{figure}

\section{Conclusion}

We have presented a non-parametric estimator of CMI (and MI) for
arbitrary combinations of discrete, continuous, and mixed variables.
Under mild assumptions, the estimator is consistent,
and on empirical simulations, the proposed estimator performs better
over other similar estimators in all sample-sizes.
Yet, it is very easy to understand how the estimate is calculated.

The development of this estimator was primarily motivated by data from
scientific applications.
There are clear advantages of using information in contrast to regression, for
example, for scientific inquiry that we reiterate.  While this method does
require independent, identically distributed data, it does not require
parametric assumptions about variable distributions or specific functional
relationships between variables such as linearity to quantify dependence.  Due
to the data processing inequality, greater shared information among random
variables indicates closer causal proximity in causal chains.  In this vein,
CMI (or MI) estimates close to or equal to zero indicate likely conditional (or
marginal) independence.  For these reasons, information is ideal for inference
and discovery causal of relationships.  And, like regression, information is
easily interpretable: CMI, $I(X;Y|Z)$, can be understood as the degree of
association or statistical dependence shared between $X$ and $Y$ given $Z$ or
controlling for $Z$.

As of now, the sampling distribution for the estimator is unknown.
Approximating the distribution with, for example, the bootstrap may be
a way forward.
However, one should note that the local estimates for a point,
$\xi_i$, is associated with its neighbors.
A more analytic approximation to the sampling distribution of
$\hat{I}$ is an interesting, valuable open problem for further
inference including testing and confidence intervals.

A lot can be done without knowing the sampling distribution as well.
Many machine learning algorithms perform tasks such a feature
selection, structure learning, and clustering using properties of CMI
without knowing its sampling distribution.
The proposed method of CMI estimation makes these algorithms
accessible for application fields.

While not ideal that this estimator is biased toward zero in high dimensions,
this knowledge will help create algorithms that can account for this weakness
or even exploit it.  Finally, we encourage others to continue researching
innovative methodologies to accommodate fields whose data is too messy for most
current data science methodologies.


%

\appendices
\section{Proof of Theorem \ref{unbiased}}\label{consistencyProof}
\begin{proof}
  Define
  $f \equiv \frac{dP_{XY|Z}}{d\paren{P_{X|Z}\times P_{Y|Z}}}$ and for a
  random variable $W$ on $\mathcal{W}$ with probability measure, $P_W$,
  and $w\in\mathcal{W}$, set
  \begin{equation}P_W(w,r) = P_W\paren{\curly{v\in\mathcal{W}:
        \norm{v-w}_\infty\leq r}}.\end{equation}
  Let $(x_1,y_1,z_1), \hdots, (x_n,y_n,z_n)$ be an i.i.d. random sample
  from $P_{XYZ}$ and that $\hat I_n(X,Y|Z)$ is the value of
  $\hat I_{\text{prop}}(X,Y|Z)$ for this sample.

  Partition $\cX\times\cY\times\cZ$ into three disjoint sets:
  \begin{enumerate}
  \item $\Omega_1 = \curly{(x,y,z): f = 0}$
  \item $\Omega_2 = \curly{(x,y,z): f>0, P_{XYZ}(x,y,z,0)>0}$
  \item $\Omega_3 = \curly{(x,y,z): f>0, P_{XYZ}(x,y,z,0)=0}$
  \end{enumerate}
  so that $\cX\times\cY\times\cZ = \Omega_1\cup\Omega_2\cup\Omega_3$.
  Notice that
  \begin{equation}\E\brac{\hat I_n(X,Y|Z)} =
    \E\brac{\frac{1}{n}\sum_{i=1}^n\xi_i} =
    \E\brac{\xi_1}\end{equation}
  so that we only need show that $\E[\xi_i]\rightarrow I(X;Y|Z)$ for one
  point.
  In light of this, we drop the subscript.
  Using the law of total expectation and properties of integrals,
  \begin{align*}
    &\abs{\E[\xi] - I(X;Y|Z)}\\
    &= \abs{\E_{XYZ}\brac{\E[\xi|X,Y,Z]} - \int f(x,y,z)dP_{XYZ}(x,y,z)}\\
    &\leq \int\abs{\E[\xi|x,y,z] - f(x,y,z)}dP_{XYZ}(x,y,z)\\
    &= \int_{\Omega_1}\abs{\E[\xi|x,y,z] - f(x,y,z)}dP_{XYZ}(x,y,z)\\
    &\quad +\int_{\Omega_2}\abs{\E[\xi|x,y,z] - f(x,y,z)}dP_{XYZ}(x,y,z)\\
    &\quad +\int_{\Omega_3}\abs{\E[\xi|x,y,z] - f(x,y,z)}dP_{XYZ}(x,y,z).
  \end{align*}
  For clarification, the value of $\E[\xi|X,Y,Z]$ depends on both the
  value of the value of the random vector $(X,Y,Z)$ and rest of the sample.
  We show that $\int_{\Omega_i}\abs{\E[\xi|x,y,z] -
    f(x,y,z)}dP_{XYZ}\rightarrow 0$ for each $i = 1,2,3$ in three cases.

  \textbf{Case1:} $(x,y,z)\in\Omega_1$.
  Let $\pi_{XY}(\Omega_1) = \curly{(x,y):(x,y,z)\in\Omega_1}$ be the
  projection onto the first two coordinates of $\Omega_1$.
  Using the definition of $f$ as the RN derivative,
  \begin{align*}
    &P_{XY|Z}(\pi_{XY}(\Omega_1))\\
    &= \int_{\pi_{XY}(\Omega_1)}fd(P_{X|Z}\times P_{Y|Z})\\
    &= \int_{\pi_{XY}(\Omega_1)}0d(P_{X|Z}\times P_{Y|Z}) = 0.
  \end{align*}
  Then $P_{XYZ}(\Omega_1) = (P_{XY|Z}\times P_Z)(\Omega_1) = 0$.
  So,
  \begin{equation}
    \int_{\Omega_i}\abs{\E[\xi|x,y,z] - f(x,y,z)}dP_{XYZ}= 0.
  \end{equation}

  \textbf{Case 2:} Assume $(x,y,z)\in\Omega_2$.
  This is the partition of discrete points because singleton have
  positive measure in $\cX\times\cY\times\cZ$.
  Using lemma~\ref{cmiDeriv}, we have
  \begin{equation}
    f(x,y,z) = \frac{P_{XYZ}(x,y,z,0)P_Z(x,y,z,0)}
    {P_{XZ}(x,y,z,0)P_{YZ}(x,y,z,0)}.
  \end{equation}
  Knowing the exact value of $f$ allows us to work with it directly.

  Let $\rho$ be the distance from $(x,y,z)$ to its $k$NN.
  Proceed in two cases, when $\rho = 0$ and when $\rho>0$ by writing
  the integrand as dominated by the following two terms:
  \begin{align*}
    &\abs{\E[\xi|x,y,z] - \log f(x,y,z)}\\
    &\leq \abs{\E[\xi|x,y,z, \rho>0] - \log f(x,y,z)}\P(\rho > 0)\\
    &\quad+\abs{\E[\xi|x,y,z, \rho = 0] - \log f(x,y,z)}\P(\rho = 0)\\
    &\equiv \abs{\E[\xi|\rho>0] - \log f}\P(\rho > 0)\\
    &\quad +\abs{\E[\xi|\rho = 0] - \log f}\P(\rho = 0)
  \end{align*}
  suppressing the $x,y,z$ for brevity.
  We bound $\abs{\E[\xi|\rho>0] - \log f}$ and $\P(\rho = 0)$ and show that
  $\abs{\E[\xi|\rho = 0] - \log f}$ and $\P(\rho > 0)$ converge to
  zero.
  By proposition~\ref{discFin}, there exist a finite set of points with
  positive measure $E\subseteq\Omega_2$ such that
  \begin{equation}P_{XYZ}(\Omega_2\backslash E) <
    \frac{\epsilon}{3(4\log n + \log C)}.\end{equation}

  Starting with $\P(\rho > 0)$,
  $\rho>0$ when less than $k$ points in the sample equal $(x,y,z)$.
  The number of points exactly equal to $(x,y,z)$ has a binomial
  distribution with parameters, $n-1$ and $P_{XYZ}(x,y,z,0)\equiv P_{XYZ}(0)$,
  $\text{Binomial}(n-1,P_{XYZ}(0))$.
  Because $\frac{k}{n}\rightarrow 0$ as $n\rightarrow\infty$, there
  must be an $n$ sufficiently large such that
  \begin{align*}
    &\max\curly{\frac{k}{n},
    \frac{-2}{n}\log\paren{\frac{\epsilon}{3(4\log n + \log C)|E|}}
    + \frac{2k}{n}}\\
    &\leq \min_{(x,y,z)\in E}P_{XYZ}(x,y,z,0).
  \end{align*}
  This inequality ensures that
  $k-1\leq (n-1)P_{XYZ}(x,y,z,0)$ for all $(x,y,z)\in E$ to use
  Chernoff's inequality \cite[\S 2.2]{boucheron2013concentration}:
  \begin{align*}
    \P(\rho>0)
    &= \P(\text{Binomial}(n-1,P_{XYZ}(0))\leq k-1)\\
    &\leq \exp\curly{\frac{-[(n-1)P_{XYZ}(0) - (k-1)]^2}
    {2P_{XYZ}(x,y,z,0)(n-1)}}\\
    &\leq \exp\curly{-\paren{\frac{1}{2}nP_{XYZ}(0) - k}}\\
    &\leq \frac{\epsilon}{3(4\log n +\log C)|E|}.
  \end{align*}
  To bound $\abs{\E[\xi|\rho>0] - \log f}$, first notice that
  $\tilde k, n_{XZ}, n_{YZ}, n_Z\leq n$.
  If $k=\tilde k$, then $\xi$ uses $\psi$ and if If $k<\tilde k$, then
  $\xi$ uses $\log$, so that
  $\abs{\xi} \leq \max\curly{4\psi(n), 4\log n} = 4\log n$.
  And, $f\leq C$ by assumption so
  $\abs{\E[\xi|\rho>0] - \log f} < 4\log n + \log C$.

  Now we show that $\abs{\E[\xi|\rho = 0] - \log f}\rightarrow 0$.
  When $\rho=0$, there must be $k$ or more points exactly equal to
  $(x,y,z)$.
  Because a point in the sample being equal to $(x,y,z)$ is an independent,
  Bernoulli event, and because when $\rho=0$, $\tilde k$, defined in
  (\ref{newtildek}), will be the total number of points equal to
  $(x,y,z)$, $\tilde k - k\sim \text{Binomial}(n - k - 1, P_{XYZ}(0))$.
  We can make identical arguments for $n_{XZ}-k, n_{YZ}-k$, and $n_Z-k$ in
  their respective subspaces so that $n_{XZ}-k\sim \text{Binomial}(n - k - 1,
  P_{XZ}(0))$, $n_{YZ}-k\sim \text{Binomial}(n - k - 1, P_{YZ}(0))$, and
  $n_Z-k\sim \text{Binomial}(n - k - 1, P_Z(0))$.
  \cite[Lemma B.2]{gao2017estimating} provides a rigorous proof for
  this.

  Showing that $\abs{\E[\xi|\rho = 0] - \log f}\rightarrow 0$,
  we can choose $k$ and $n$ sufficiently large so that
  $\frac{1}{k}\leq \frac{\epsilon}{48|E|}$, $k\geq
  \frac{P_Z(0)}{1-P_Z(0)}$ and
  $\frac{k}{n}\leq \frac{\epsilon}{24|E|}$.
  Assume $\tilde k = k$, so that $\xi$ will use $\psi$.
  Using lemma~\ref{discBound} four times and that
  $P_{XYZ}(0)\leq P_{XZ}(0), P_{YZ}(0)\leq P_Z(0)$,
  \begin{align*}
    &\abs{\E[\xi|\rho = 0] - \log f}\\
    &= \Big|\E[\psi(\tilde k)|\rho = 0] - \E[\psi(n_{XZ})|\rho = 0]\\
    &\quad - \E[\psi(n_{YZ})|\rho = 0] + \E[\psi(n_{Z})|\rho = 0]\\
    &\quad - \log \frac{(nP_{XYZ}(0))(nP_Z(0))}{(nP_{XZ}(0))(nP_{YZ}(0))}\Big|\\
    &\leq \abs{\E[\psi(\tilde k)|\rho = 0] - \log nP_{XYZ}(0)}\\
    &\quad + \abs{\E[\psi(n_{XZ})|\rho = 0] - \log nP_{XZ}(0)}\\
    &\quad + \abs{\E[\psi(n_{YZ})|\rho = 0] - \log nP_{YZ}(0)}\\
    &\quad + \abs{\E[\psi(n_{Z})|\rho = 0] - \log nP_{Z}(0)}\\
    &\leq \frac{2}{k} + \frac{k}{nP_{XYZ}(0)}
      + \frac{2}{k} + \frac{k}{nP_{XZ}(0)}\\
    &\quad+ \frac{2}{k} + \frac{k}{nP_{YZ}(0)}
      + \frac{2}{k} + \frac{k}{nP_{Z}(0)}\\
    &\leq \frac{8}{k} + \frac{4k}{nP_{XYZ}(0)}\\
    &\leq \frac{\epsilon}{6|E|} + \frac{\epsilon}{6|E|P_{XYZ}(0)}.
  \end{align*}
  If $\tilde k > k$, $\xi$ will use $\log$ and rather than $\psi$.
  Lemma~\ref{discBound} shows that the bound used above will also work
  in this case.
  
  It is clear that $\P(\rho = 0)\leq 1$.

  Putting together the previous parts,
  \begin{align*}
    &\int_{\Omega_2}\abs{\E[\xi|x,y,z] - \log f(x,y,z)}dP_{XYZ}(x,y,z)\\
    &= \sum_{(x,y,z)\in\Omega_2}
      \abs{\E[\xi|x,y,z] - \log f(x,y,z)}P_{XYZ}(x,y,z,0)\\
    &\equiv \sum_{(x,y,z)\in\Omega_2}
      \abs{\E[\xi] - \log f}P_{XYZ}(0)\\
    &= \sum_{(x,y,z)\in E}
      \abs{\E[\xi] - \log f}P_{XYZ}(0)\\
    &\quad + \sum_{(x,y,z)\in\Omega_2\backslash E}
      \abs{\E[\xi] - \log f}P_{XYZ}(0)\\
    &\leq \sum_{(x,y,z)\in E}
      \abs{\E[\xi|\rho>0] - \log f}\P(\rho > 0)P_{XYZ}(0)\\
    &\quad + \sum_{(x,y,z)\in E}
      \abs{\E[\xi|\rho = 0] - \log f}\P(\rho = 0)P_{XYZ}(0)\\
    &\quad + \sum_{(x,y,z)\in\Omega_2\backslash E}
      \abs{\E[\xi] - \log f}P_{XYZ}(0)\\
    &\leq \sum_{(x,y,z)\in E} \log(n^4C)
      \paren{\frac{\epsilon}{3\log(n^4C)|E|}}P_{XYZ}(0)\\
    &\quad + \sum_{(x,y,z)\in E}
      \paren{\frac{\epsilon}{6|E|} +
      \frac{\epsilon}{6|E|P_{XYZ}(0)}}P_{XYZ}(0)\\
    &\quad + \sum_{(x,y,z)\in\Omega_2\backslash E}
      \paren{4\log n + \log C}P_{XYZ}(0)\\
    &\leq \log(n^4C)|E|
      \paren{\frac{\epsilon}{3(\log(n^4C))|E|}}
      + |E|\paren{\frac{\epsilon}{3|E|}}\\
    &\quad + P_{XYZ}\paren{\Omega_2\backslash E}
      \paren{4\log n + \log C}\\
    &= \frac{\epsilon}{3} + \frac{\epsilon}{3}
      + \paren{\frac{\epsilon}{3\log(n^4C)}}\log(n^4C)\\
    &= \epsilon.
  \end{align*}

  \textbf{Case 3:} Assume $(x,y,z)\in\Omega_3$.
  This is the continuous partition because singletons have
  zero measure in $\cX\times\cY\times\cZ$.
  Lemma~\ref{ktildelim} assures that $\tilde k\rightarrow k$ almost
  surely as $n\rightarrow \infty$; $P_{XYZ}\paren{\curly{(x,y,z)\in\Omega_3:
      \tilde k\rightarrow k}} = 1$.
  $\tilde k$ is discrete so there is an $N$ such that for $n\geq N$,
  $\tilde k = k$ with probability one.
  
  Define $F_\rho(r)$ as the cumulative distribution function of the
  $k$NN distance, $r$;
  that is, $F_\rho(r)$ is the probability that that the $k$NN distance
  is $r$ or less.
  Begin by decomposing the integrand into its parts:
  \begin{align}
    &\abs{\E\brac{\xi|X,Y,Z} - \log f(X,Y,Z)}\label{contInt}\\
    &\equiv\abs{\E\brac{\xi} - \log f}\nonumber\\
    &= \abs{\int_0^\infty\paren{\E\brac{\xi|\rho = r} - \log
      f}dF_\rho(r)}\nonumber\\
    &= \Big|\int_0^\infty\E\brac{\xi|\rho = r}
      -\log\paren{\frac{P_{XYZ}(r)P_Z(r)}{P_{XZ}(r)P_{YZ}(r)}}\nonumber\\
    &\quad+\log\paren{\frac{P_{XYZ}(r)P_Z(r)}{P_{XZ}(r)P_{YZ}(r)}}
      -\log fdF_\rho(r)\Big|\nonumber\\
    &= \Big|\int_0^\infty\E\brac{\psi(k) - \psi(n_{XZ}) - \psi(n_{YZ})
      - \psi(n_Z)|\rho = r}\nonumber\\
    &\quad -\log\paren{\frac{(nP_{XYZ}(r))(nP_Z(r))}{(nP_{XZ}(r))(nP_{YZ}(r))}}\nonumber\\
    &\quad+\log\paren{\frac{P_{XYZ}(r)P_Z(r)}{P_{XZ}(r)P_{YZ}(r)}}
      -\log fdF_\rho(r)\Big|\nonumber\\    
    &\leq \abs{\int_0^\infty \psi(k) - \log(nP_{XYZ}(r))dF_\rho(r)}\label{intk}\\
    &\quad+ \abs{\int_0^\infty \E[\psi(n_{XZ})|\rho = r] -
      \log(nP_{XZ}(r))dF_\rho(r)}\label{intnxz}\\
    &\quad+\abs{\int_0^\infty \E[\psi(n_{YZ})|\rho = r] -
      \log(nP_{YZ}(r))dF_\rho(r)}\label{intnyz}\\
    &\quad+\abs{\int_0^\infty \E[\psi(n_Z)|\rho = r] -
      \log(nP_Z(r))dF_\rho(r)}\label{intnz}\\ 
    &\quad+\abs{\int_0^\infty\log\paren{\frac{P_{XYZ}(r)P_Z(r)}{P_{XZ}(r)P_{YZ}(r)}}
      -\log fdF_\rho(r)}\label{intf}.
  \end{align}

  Next, we show that with sufficiently large $n$, each of these terms is less
  than $\epsilon/5$.
  Do to this, we change variables for each integral using
  lemma~\ref{distDeriv}.

  Beginning with (\ref{intk}),
  \begin{align*}
    &\abs{\int_0^\infty \psi(k) - \log(nP_{XYZ})(r)dF_\rho(r)}\\
    &= \abs{\psi(k) - \log n -\int_0^\infty \log P_{XYZ}(r)dF_\rho(r)}\\
    &= \left| \psi(k) - \log n - 
      \int_0^\infty P_{XYZ}(r)  \frac{(n-1)!}{(k-1)!(n-k-1)!}\right.\\
    &\quad \left.\vphantom{\int_0^\infty}[P_{XYZ}(r)]^{k-1}
      [1-P_{XYZ}(r)]^{n-k-1}dP_{XYZ}(r)\right|\\
    &= \left| \psi(k) - \log n - 
      \frac{(n-1)!}{(k-1)!(n-k-1)!}\vphantom{\int_0^\infty} \right.\\
    &\quad \left.\int_0^\infty[P_{XYZ}(r)]^{k}
      [1-P_{XYZ}(r)]^{n-k-1}dP_{XYZ}(r)\right|\\
    &= \abs{\psi(k) - \log n - (\psi(k) - \psi(n))}\\
    &= \abs{\psi(n) - \log(n)} < \frac{1}{n}.
  \end{align*}
  For lines~\ref{intnxz},~\ref{intnyz}, and~\ref{intnz}, consider the
  random variables $n_{XZ}, n_{YZ}$, and $n_Z$ defined in
  line~\ref{nwi}.
  In this case, we know that $\tilde k = k$ almost surely.
  Note that $n_{XZ}, n_{YZ}, n_Z\geq k$.
  Observation $j$ in the sample will contribute to the count of
  $n_{W,i} - k$ for $W\in\curly{(XZ), (YZ), Z}$ when
  $\norm{w_i-w_j}_\infty\leq \rho_{k,i}$ given that it is not
  one of the first $k$ nearest neighbors.
  There are $n - k - 1$ independent, identically distributed, data
  points left not counting the $k$ nearest neighbors or point $i$.
  A point $j$ has probability, $P_W(\rho_{k,i})$, that it is
  within a radius of $\rho_{k,i}$ in the $W$ subspace.
  The probability that a point falls within a radius of
  $\rho_{k,i}$ in the $XYZ$-space is
  $P_{XYZ}(\rho_{k,i})$.
  Using basic conditional probability rules, one can see that the 
  probability that any point contributes to the count of $n_W$ is
  $\frac{P_W(\rho_{k,i}) - P_{XYZ}(\rho_{k,i})}
  {1 - P_{XYZ}(\rho_{k,i})}$.
  Then, for $W\in\curly{XZ,YZ,Z}$
  \begin{equation}\begin{split}
      &n_{W,i}-k\sim\\
      &\quad\text{Binomial}\paren{n-k-1, \frac{P_W(\rho_{k,i}) -
          P_{XYZ}(\rho_{k,i})}{1-P_{XYZ}(\rho_{k,i})}}
    \end{split}\end{equation}
  $P_{XYZ}(\rho_{k,i})\leq P_W(\rho_{k,i})$ for all
  points.
  Choosing $k$ such that $k\geq \frac{15+3\epsilon}{\epsilon}$ and
  applying lemma~\ref{contBound}, we bound 
  lines~\ref{intnxz},~\ref{intnyz}, and~\ref{intnz} by
  $\frac{\epsilon}{5}$.

  Moving to line~\ref{intf}, using lemma~\ref{rnDerivLim}, we have
  \begin{equation}\frac{P_{XYZ}(r)P_Z(r)}{P_{XZ}(r)P_{YZ}(r)}\rightarrow f\end{equation}
  (converges pointwise) as $r\rightarrow 0$ and
  \begin{equation}\label{probsBound}
    \frac{P_{XYZ}(r)P_Z(r)}{P_{XZ}(r)P_{YZ}(r)}\leq C
  \end{equation}
  almost everywhere $[P_{X|Z}\times P_{Y|Z}]$.
  Using Egoroff's theorem, there exists a measurable set, $E\subseteq
  \Omega_3$ such that
  \begin{equation}\label{eContBound}
    P_{XYZ}(\Omega_3\backslash E)\leq \frac{\epsilon}{10\log C}
  \end{equation}
  and 
  \begin{equation}\frac{P_{XYZ}(r)P_Z(r)}{P_{XZ}(r)P_{YZ}(r)}\xrightarrow{U} f\end{equation}
  (converges uniformly) as $r\rightarrow 0$ on $E$.
  Using the uniform convergence on $E$, there exists $r_\epsilon > 0$
  such that for all $r\leq r_\epsilon$
  \begin{equation}
    \abs{\log \frac{P_{XYZ}(r)P_Z(r)}{P_{XZ}(r)P_{YZ}(r)} - \log f}
    \leq \frac{\epsilon}{20}
  \end{equation}
  for all $(x,y,z)\in E$.
  And for sufficiently large $n$, we have
  \begin{equation}\max\curly{\frac{k}{n},
      \frac{-2\log\paren{\frac{\epsilon}{40\log C}} + 2k}{n}}\leq
    P_{XYZ}(r_\epsilon).\end{equation}
  Consider the probability, $\P(\rho > r_\epsilon)$, that a point's
  $k$NN distance is greater than $r_\epsilon$.
  This can only happen when $k-1$ or less neighbors fall within a
  radius of $r_\epsilon$.  There are $n-1$ independent, identically
  distributed points that can potentially fall in to this region
  each with probability, $P_{XYZ}(r_\epsilon)$ so that this also has a
  binomial distribution.
  Again using Chernoff's inequality,
  \begin{align*}
    &\P(\rho > r_\epsilon)\\
    &\leq \exp\paren{\frac{-[(n-1)P_{XYZ}(r_\epsilon) - (k-1)]^2}
      {2P_{XYZ}(r_\epsilon)(n-1)}}\\
    &\leq \exp\paren{-\frac{1}{2}nP_{XYZ}(r_\epsilon) + k}\\
    &\leq \frac{\epsilon}{40\log C}.
  \end{align*}
  With assumption~\ref{fBound}, $f\leq C$, and line~\ref{probsBound} from
  proposition~\ref{rnDerivLim},
  \begin{equation}
    \abs{\log \frac{P_{XYZ}(r)P_Z(r)}{P_{XZ}(r)P_{YZ}(r)} - \log f}\leq 2\log C.
  \end{equation}
  For points $(x,y,z)\in E$,
  \begin{align*}
    &\abs{\int_0^\infty\log \frac{P_{XYZ}(r)P_Z(r)}{P_{XZ}(r)P_{YZ}(r)} -
      \log f dF_\rho(r)}\\
    &\leq\int_0^\infty\abs{\log \frac{P_{XYZ}(r)P_Z(r)}{P_{XZ}(r)P_{YZ}(r)} -
      \log f}dF_\rho(r)\\
    &= \int_0^{r_\epsilon}\abs{\log \frac{P_{XYZ}(r)P_Z(r)}{P_{XZ}(r)P_{YZ}(r)} -
      \log f}dF_\rho(r)\\
    &\quad + \int_{r_\epsilon}^\infty\abs{\log
      \frac{P_{XYZ}(r)P_Z(r)}{P_{XZ}(r)P_{YZ}(r)} -
      \log f}dF_\rho(r)\\
    &\leq \int_0^{r_\epsilon}\frac{\epsilon}{20} dF_\rho(r) +
      \int_{r_\epsilon}^\infty 2\log C dF_\rho(r)\\
    &= \frac{\epsilon}{20}\P(\rho \leq r_\epsilon) + (2\log C) \P(\rho >
      r_\epsilon)\\
    &\leq \frac{\epsilon}{20} + \frac{\epsilon}{20} = \frac{\epsilon}{10}.
  \end{align*}
  
  But, for points $(x,y,z)\in\Omega_3\backslash E$, it is only necessary
  bound the integrand,
  \begin{align*}
    &\abs{\int_0^\infty\log \frac{P_{XYZ}(r)P_Z(r)}{P_{XZ}(r)P_{YZ}(r)} -
      \log f dF_\rho(r)}\\
    &\leq\int_0^\infty\abs{\log \frac{P_{XYZ}(r)P_Z(r)}{P_{XZ}(r)P_{YZ}(r)} -
      \log f}dF_\rho(r)\\
    &\leq\int_0^\infty 2\log C dF_\rho(r)\\
    &\leq 2\log C.
  \end{align*}
  The last step follows because $F_\rho(r)$ is a probability measure.
  Integrating term~\ref{intf} over all of $\Omega_3$,
  \begin{align*}
    &\int_{\Omega_3}\left|\int_0^\infty
      \log\frac{P_{XYZ}(r)P_Z(r)}{P_{XZ}(r)P_{YZ}(r)}
     -\log fdF_\rho(r)\right|dP_{XYZ}\\
    &\leq \int_{E}\left|\int_0^\infty
      \log\frac{P_{XYZ}(r)P_Z(r)}{P_{XZ}(r)P_{YZ}(r)}
      -\log fdF_\rho(r)\right|dP_{XYZ}\\
    &+\int_{\Omega_3\backslash E}\left|\int_0^\infty
      \log\frac{P_{XYZ}(r)P_Z(r)}{P_{XZ}(r)P_{YZ}(r)}
      -\log fdF_\rho(r)\right|dP_{XYZ}\\
    &\leq \int_E \frac{\epsilon}{10}dP_{XYZ} +
      \int_{\Omega_3\backslash E} 2\log C dP_{XYZ}\\
    &= \frac{\epsilon}{10} + (2\log C) P_{XYZ}(\Omega\backslash E)
      \leq \frac{\epsilon}{5}
  \end{align*}
  where we used Ergoroff's theorem from line~\ref{eContBound} in the
  last line.
  Now we integrate line~\ref{contInt} over $\Omega_3$ using the
  previous arguments showing that lines~\ref{intk}~\ref{intnxz},
  ~\ref{intnyz},~\ref{intnz},~and~\ref{intf} are all bounded.
  Choosing $n$ large enough to satisfy the previous conditions,
  we have

  \begin{align*}
    &\int_{\Omega_3} \abs{\E\brac{\xi} - \log f}dP_{XYZ}\\
    &\leq \int_{\Omega_3}\abs{\int_0^\infty \psi(k) -
      \log(nP_{XYZ}(r))dF_\rho}dP_{XYZ}\\
    &+ \int_{\Omega_3}\left|\int_0^\infty \E[\psi(n_{XZ})]
    - \log(nP_{XZ}(r))dF_\rho\right|dP_{XYZ}\\
    &+\int_{\Omega_3}\left|\int_0^\infty \E[\psi(n_{YZ})]
     - \log(nP_{YZ}(r))dF_\rho\right|dP_{XYZ}\\
    &+\int_{\Omega_3}\left|\int_0^\infty \E[\psi(n_Z)]
     - \log(nP_Z(r))dF_\rho\right|dP_{XYZ}\\ 
    &+\int_{\Omega_3}\left|\int_0^\infty
      \log\frac{P_{XYZ}(r)P_Z(r)}{P_{XZ}(r)P_{YZ}(r)}
     -\log fdF_\rho\right|dP_{XYZ}\\
    &\leq \int_{\Omega_3}\frac{\epsilon}{5}dP_{XYZ}+ \int_{\Omega_3}
      \frac{\epsilon}{5}dP_{XYZ}\\
    &\quad+\int_{\Omega_3}\frac{\epsilon}{5}dP_{XYZ}+\int_{\Omega_3}
      \frac{\epsilon}{5}dP_{XYZ} + \frac{\epsilon}{5}\\
    &=\epsilon
  \end{align*}
\end{proof}

\section{Proof of Theorem \ref{variance}}\label{varProof}
\begin{proof}
  Let $ W_1'\hdots, W_n'$ be another random sample of size $n$ such
  that for each $i$, $W_i = (X_i, Y_i, Z_i)$, $W_i' = (X_i', Y_i', Z_i')$ and
  that $W_i \stackrel{d}{=} W_i'$ (equally distributed).  Let $W^{(i)} =
  \{W_1,\hdots, W_{i-1}, W_i', W_{i+1}, \hdots, W_n\}$ and let
  $W^{i-} = \{W_1,\hdots, W_{i-1}, W_{i+1}, \hdots, W_n\}$
  We proceed using the Stein-Efron inequality as in
  \cite[Theorem 3.1]{boucheron2013concentration},
  \begin{equation*}
    \var\paren{\hat I_n(W)}
    \leq \frac{1}{2}\sum_{i=1}^n \E\brac{\hat I_n(W) - \hat
      I_n(W^{(i)})}^2.
  \end{equation*}

  To reduce the number of cases we must examine, consider the
  following supremum over possible values $w_1,\hdots w_n, w_i'$ of
  the random vector $W$:
  \begin{align*}
    &\sup_{w_1,\hdots w_n, w_i'}\abs{\hat I_n(W) - \hat I_n(W^{(i)})}\\
    &\leq \sup_{w_1,\hdots w_n, w_i'}\left(
      \abs{\hat I_n(W) - \hat I_n(W^{i-})}\right.\\
    &\quad+ \left. \abs{\hat I_n(W^{i-}) - \hat I_n(W^{(i)})}\right)\\
    &\leq \sup_{w_1,\hdots w_n}\abs{\hat I_n(W) - \hat I_n(W^{i-})}\\
    &\quad+ \sup_{w_1,\hdots, w_{i-1}, w_i',w_{i+1},\hdots, w_n}
      \abs{\hat I_n(W^{i-}) - \hat I_n(W^{(i)})}\\
    &= 2\sup_{w_1,\hdots w_n}\abs{\hat I_n(W) - \hat I_n(W^{i-})}\\
    &= \frac{2}{n}\sup_{w_1,\hdots w_n}
      \sum_{j=1}^n\abs{\xi_j(W) - \xi_j(W^{i-})}.\\
  \end{align*}
  The penultimate step holds because $W \stackrel{d}{=} W^{(i)}$.
  
  We proceed by bounding $\abs{\xi_j(W) - \xi_j(W^{i-})}$ by looking
  at the individual cases.\\
  \noindent
  \textbf{Case 1:} $i=j$.\\
  Notice that if $0< a,b\leq n$ then
  \begin{align*}
    &\abs{\psi(a) - \log(b)} \leq \abs{\psi(a) - \log(a)} +
      \abs{\log(b) - \log(b)}\\
    &\leq \frac{1}{b} + \log(\max\curly{a,b}) \leq \log n + 1.
  \end{align*}
  Using this,
  \begin{align*}
    &\abs{\xi_j(W) - \xi_j(W^{i-})}\\
    &\leq \abs{\psi(k) - \log(\tilde k_j')}
      + \abs{\psi(n_{XZ,j}) - \log(n_{XZ,j}')}\\
    &\quad + \abs{\psi(n_{YZ,j}) - \log(n_{YZ,j}')}
      + \abs{\psi(n_{Z,j}) - \log(n_{Z,j}')}\\
    &\leq 4\log n + 4.
  \end{align*}
  In the summation from $j=1$ to $n$, this can only happen one times,
  so we have that $\sum_{j=1}^n \abs{\xi_j(W) - \xi_j(W^{i-})}\leq
  4\log n + 4$. 
  
  \noindent
  \textbf{Case 2:} $i\neq j, \tilde k_j > k$.\\
  Recall that $\xi_j(W) = \log(\tilde k_j) - \log(n_{XZ,j}) -
  \log(n_{YZ,j}) + \log(n_{Z,j})$ and that
  $\rho_{k,j}$ is the $\ell_\infty$-distance from $W_j$ to its
  $k$NN.
  Removing $W_i$ from $W$ will only change $\xi_j(W)$ if $W_i$ is
  counted in $\tilde k_j, n_{XZ,j}, n_{YZ,j}$, or
  $n_{Z,j}$.
  Because $\tilde k_j > k$, there must be at least two points whose
  distance to $W_j$ is exactly $\rho_{k,j}$, so removing one
  point cannot change $\rho_{k,j}$, regardless of its location
  with respect to $W_j$.
  Because $\rho_{k,j}$ will remain unchanged after removing
  $W_i$ from $W$, $\tilde k_j, n_{XZ,j}, n_{YZ,j}$, or
  $n_{Z,j}$ can each only decrease by a count of one.
  Under $\xi_j(W^{i-})$, if $\tilde k_j = k$, then the $log$ function
  will become $\psi$. 
  In general, we have that $\psi(w)-\psi(w-1) = \frac{1}{w-1}$,
  $\log(w) - \log(w-1) = \log\paren{\frac{w}{w-1}} \leq \frac{1}{w-1}$
  and, $\log(w) - \psi(w-1) = \log(w) - \psi(w) +\frac{1}{w-1}\leq \frac{2}{w-1}$.
  Regardless, we have
  \begin{align*}
    &\abs{\xi_j(W) - \xi_j(W^{i-})}\\
    &\leq \abs{\log(\tilde k_j) - \psi(\tilde k_j-1)}\\
    &\quad+\abs{\log(n_{XZ,j}) - \psi(n_{XZ,j}-1)}\\
    &\quad+ \abs{\log(n_{YZ,j}) - \psi(n_{YZ,j}-1)}\\
    &\quad+ \abs{\log(n_{Z,j}) - \psi(n_{Z,j}-1)}\\
    &\leq \frac{2}{\tilde k_j - 1}
      +\frac{2}{n_{XZ,j} - 1}\\
    &\quad+\frac{2}{n_{YZ,j} - 1}
      + \frac{2}{n_{Z,j} - 1}.
  \end{align*}

  Now, rather than considering the number of points that can change
  with the removal of $W_i$, we focus on the number of counts,
  $\tilde k_j, n_{XZ,j}, n_{YZ,j}$, and 
  $n_{Z,j}$, that will change.
  If $W_i$ is among the $\tilde k_j$NN of $W_j$, then its removal can
  change at most the $\tilde k_j$ points within a distance of
  $\rho_{k,j}$ in all coordinates.
  If $W_i$ is not among the $\tilde k_j$NN of $W_j$ but is counted in
  $n_{XZ,j}$, (and possibly in $n_{Z,j}$ too), then 
  its removal will not affect $\tilde k_j$ or $n_{YZ,j}$ and
  will only change $n_{XZ,j}$, (and $n_{Z,j}$) for the
  points within a distance of $\rho_{k,j}$ from $W_j$ in the
  $XZ$ coordinates, which is $n_{XZ,j}$.
  Similarly, $n_{YZ,j}$ and $n_{Z,j}$ will change for
  at most $n_{YZ,j}$ and $n_{Z,j}$ points, respectively.
  So, we have
  \begin{align*}
    &\sum_{j=1}^n\abs{\xi_j(W) - \xi_j(W^{(i)})}\\
    &\leq \sum_{j=1}^n \frac{2}{\tilde k_j - 1}
      +\sum_{j=1}^n\frac{2}{n_{XZ,j} - 1}\\
    &\quad+\sum_{j=1}^n\frac{2}{n_{YZ,j} - 1}
      + \sum_{j=1}^n\frac{2}{n_{Z,j} - 1}\\
    &\leq \frac{2\tilde k_j}{\tilde k_j - 1}
      +\frac{2n_{XZ,j}}{n_{XZ,j} - 1}\\
    &\quad+\frac{2n_{YZ,j}}{n_{YZ,j} - 1}
      + \frac{2n_{Z,j}}{n_{Z,j} - 1}\\
    &\leq 16
  \end{align*}

  \noindent
  \textbf{Case 3:} $i\neq j, \tilde k_j = k$.\\
  Again, removing $W_i$ from $W$ will change $\xi_j(W)$ only if
  $W_i$ is counted in at least one of $\tilde k_j, n_{XZ,j},
  n_{YZ,j}$, or $n_{Z,j}$.
  If $W_i$ is within the $k$NN of $W_j$, then removing $W_i$ will
  change the value of $\rho_{k,j}$.
  Because $\rho_{k,j}$ is different, we cannot say how
  $n_{XZ,j}, n_{YZ,j}$, or $n_{Z,j}$ will change
  so we give the loosest bound from case 1:
  \begin{equation*}
    \abs{\xi_j(W) - \xi_j(W^{i-})}
    \leq 4\log n + 4.
  \end{equation*}

  Using the first part of \cite[Lemma C.1]{gao2017estimating},
  if $U_i', U_1, \hdots, U_n$ are vectors in $\R^d$ and
  $\mathbf{U} = \curly{U_1,\hdots, U_{j-1},U_i', U_{j+1},\hdots, U_n}$,
  then 
  \begin{equation*}
    \sum_{j=1}^n I_{\curly{\text{$U_i'$ is in the $k$NN of 
          $U_j$ in $\mathbf{U}$}}}\leq k\gamma_d
  \end{equation*}
  where $\gamma_d$ is a constant that only depends on the dimension of
  the $XYZ$ space \cite[Corollary 6.1]{gyorfi2006distributionfree}.
  With this, we have
  \begin{equation*}
    \sum_{i=1}^n \brac{\xi_j(W) - \xi_j(W^{i-})} \leq k\gamma_d(4\log
    n + 4). 
  \end{equation*}

  If $W_i$ is not within the $k$NN of $W_j$, it can still contribute
  to the count of $n_{XZ,j}, n_{YZ,j}$, or
  $n_{Z,j}$.
  In this case $\rho_{k,j}$ will not change, so removing one
  point will decrease $n_{XZ,j}, n_{YZ,j}$, or
  $n_{Z,j}$ by at most one, similar to case 2.
  \begin{align*}
    &\abs{\xi_j(W) - \xi_j(W^{(i)})}\\
    &\leq \abs{\psi(k) - \psi(k)}\\
    &\quad+\abs{\psi(n_{XZ,j}) - \psi(n_{XZ,j}-1)}\\
    &\quad+ \abs{\psi(n_{YZ,j}) - \psi(n_{YZ,j}-1)}\\
    &\quad+ \abs{\psi(n_{Z,j}) - \psi(n_{Z,j}-1)}\\
    &= \frac{1}{n_{XZ,j} - 1} +\frac{1}{n_{YZ,j} - 1}\\
     &\quad + \frac{1}{n_{Z,j} - 1}.
  \end{align*}

  Using the second part of \cite[Lemma C.1]{gao2017estimating},
  if $U_i', U_1, \hdots, U_n$ are vectors in $\R^d$ and
  $\mathbf{U} = \curly{U_1,\hdots, U_{j-1},U_i', U_{j+1},\hdots, U_n}$,
  then 
  \begin{equation*}
    \sum_{j=1}^n\frac{1}{k_i}I_{\curly{\text{$U_i'$ is in the $k_i$NN of 
          $U_j$ in $\mathbf{U}$}}}\leq \gamma_d(\log n + 1).
  \end{equation*}
  Then
  \begin{align*}
    &\sum_{j=1}^n\abs{\xi_j(W) - \xi_j(W^{i-})}\\
    &\leq \sum_{j=1}^n\frac{1}{n_{XZ,j} - 1}
    +\sum_{j=1}^n\frac{1}{n_{YZ,j} - 1}\\
      &\quad+ \sum_{j=1}^n\frac{1}{n_{Z,j} - 1}\\
    &\leq \sum_{j=1}^n\frac{1}{n_{XZ,j}}
    +\sum_{j=1}^n\frac{1}{n_{YZ,j}}\\
    &\quad+ \sum_{j=1}^n\frac{1}{n_{Z,j}} + 3\\
    &\leq (\gamma_{d_{XZ}})(\log n + 1) + \gamma_{d_{YZ}}(\log n +
      1)\\ 
      &\quad + \gamma_{d_{Z}}(\log n + 1) + 3\\
    &\leq \gamma_{d}(\log n + 1) + 3
  \end{align*}
  where $d_{XZ}$ is the dimension of $XZ$, etc. 

  Combining all of these cases, we have
  \begin{align*}
    &\sum_{j=1}^n\abs{\xi_j(W) - \xi_j(W^{(i)})}\\
    &\leq (4\log n + 4) + 16 + k\gamma_d (4\log n + 4)\\
    &\quad+ \gamma_{d}(\log n + 1) + 3\\
    &\leq 36k\gamma_d\log n
  \end{align*}
  for $n\geq 2,k\geq 1$ (and $d\geq 3$ so $\gamma_d\geq 3$).
  Using Stein-Efron inequality,
  \begin{align*}
    \var
    &\paren{\hat I_n(W)}\\
    &\leq \frac{1}{2}\sum_{i=1}^n \E\brac{\hat I_n(W) - \hat
      I_n(W^{(i)})}^2\\
    &= \frac{1}{2}\sum_{i=1}^n \E\abs{\frac{1}{n}\sum_{j=1}^n \xi_j(W) -
      \frac{1}{n}\sum_{j=1}^n \xi_j(W^{(i)})}^2\\
    &\leq \frac{1}{2n^2}\sum_{i=1}^n \E\brac{\sum_{j=1}^n 
      \abs{\xi_j(W) - \xi_j(W^{(i)})}}^2\\
    &\leq\frac{1}{2n^2}\sum_{i=1}^n \E\brac{\sum_{j=1}^n 
      \sup_{W}\abs{\xi_j(W) - \xi_j(W^{i-})}}^2\\
    &\leq\frac{1}{2n^2}\sum_{i=1}^n \E\brac{ 
      36k\gamma_d\log n}^2\\
    &= \frac{648k^2\gamma_d^2(\log n)^2}{n}\\
    &\rightarrow 0.
  \end{align*}
  The last step uses l'Hospital's rule twice.
\end{proof}

\section{Proof of Corollary \ref{conv}}\label{convProof}
\begin{proof}
  From the proof in theorem \ref{variance}, it is easy to verify that
  \begin{align*}
    \sup_{w_1,\hdots w_n, w_i'}&\abs{\hat I_n(W) - \hat I_n(W^{(i)})}\\
    &\leq \frac{2}{n}\sup_{w_1,\hdots w_n}
      \sum_{j=1}^n\abs{\xi_j(W) - \xi_j(W^{i-})}\\
    &\leq \frac{72k\gamma_d\log n}{n}.
  \end{align*}
  So, $\hat I_n$ satisfies the bounded difference property.
  Using the bounded difference inequality
  (\cite[Theorem 6.2]{boucheron2013concentration}) with
  \[v = \frac{1296k^2(\log n)^2}{n},\]
  we bound the one-sided probability by $\exp\curly{-t^2/(2v)}$ and
  simply multiply this value by a factor of 2.
\end{proof}

\section{Proof of Theorem \ref{hiDim}}\label{hiDimProof}
\begin{proof}
  Let $(x,y,z)$ be an arbitrary point in the domain of $(X,Y,Z)$.
  Choose $r\geq 0$ if $(x,y,z)$ is a discrete point and $r>0$ if
  $(x,y,x)$ is a continuous point.
  Recall that we define $P_Z(r)\equiv P_Z(B(z,r))$.
  Proceeding by contradiction, assume that
  $\lim_{d\rightarrow\infty}P_{Z}(r)>0$; that is, there exists a
  $\delta>0$   such that for every $D >0$, there is a $d \geq D$ such
  that $P_{Z}(r)>\delta$.
  $B(z,r)$ is a $d$-dimensional, $\ell_\infty$-ball so it can be
  written as the product of $d$ sets.
  Defining $Z^k\equiv (Z_{k-1},\cdots, Z_1)$ for $k=1,2,\dots,d$.
  $P_{Z_k|Z^k}(r) \equiv \P(Z_k\in \pi_k(B(z,r))| Z^k\in\pi^k(B(z,r)))$
  where $\pi_k$ is the projection on to the $k$th coordinate and
  $\pi^k$ is the projection on to the $k-1,\dots,1$ coordinates.
  Then we have that
  \begin{equation*}
    \prod_{k=1}^d P_{Z_k|Z^k}(r) = P_{Z}(r)> \delta.
  \end{equation*}
  Then
  \[\lim_{d\rightarrow\infty}\sum_{k=1}^d \log P_{Z_k|Z^k}(r) > \log
    \delta > -\infty.\]
  For each $k$, $\log P_{Z_k|Z^k}(r)\leq 0$ so 
  $\log P_{Z_k|Z^k}(r)\rightarrow 0$ as $d\rightarrow\infty$ using the
  fact that $a_i\geq 0, \sum_{i=1}^\infty a_i < M$ for some
  $M\Rightarrow a_i\rightarrow 0$.

  Choose $\epsilon>0$ and 
  let $\mathcal{Q}$ be a finite partition of the domain of $Z$ into
  sets with positive measure in $P_Z$.
  Because $z$ and $r$ were chosen arbitrarily in the previous part,
  then for each $Q\in\mathcal{Q}$, there is a point $z_Q$ in the
  domain of $Z$ and distance $r_Q$ such that
  $B(z_Q,r_Q)\subseteq Q$.
  Then there must be a $d_Q$ such that for every $k\geq d_Q$, $-\log
  P_{Z_k|Z^K}(B(z_Q,r_Q))\leq \frac{\epsilon}{\norm{\mathcal{Q}}}$
  because $\log P_{Z_k|Z^K}(B(z_Q,r_Q))\rightarrow 0$ for each $Q$.
  Choosing $k\geq \max_{Q\in\mathcal{Q}}d_Q$, we have that 
  \begin{align*}
    &\sum_{Q\in\mathcal{Q}} -P_{Z_kZ^k}(Q)\log P_{Z_k|Z^k}(Q)\\
    &\leq \sum_{Q\in\mathcal{Q}} -\log P_{Z_k|Z^k}(Q)\\
    &\leq \sum_{Q\in\mathcal{Q}} -\log  P_{Z_k|Z^K}(B(z_Q,r_Q))\\
    &\leq \sum_{Q\in\mathcal{Q}}\frac{\epsilon}{\norm{\mathcal{Q}}}
      = \epsilon.
  \end{align*}

  Let $\curly{\mathcal{Q}_l: l=1,2,\dots}$ be a sequence of
  increasingly fine partitions of the domain of $Z$ into
  sets with positive measure in $P_Z$.
  Using \cite[Lemma 7.18]{gray2011entropy},
  we have that
  \[H(Z_k|Z^k) = \lim_{l\rightarrow\infty}\sum_{Q\in\mathcal{Q}_l}
    -P_{Z_kZ^k}(Q)\log P_{Z_k|Z^k}(Q)\leq \epsilon.\]
  Using Ces\`aro's lemma ($a_i\rightarrow a\Rightarrow
  \frac{1}{n}\sum_{i=1}^na_i\rightarrow a$),
  \[\lim_{d\rightarrow\infty}\frac{1}{d}H(Z) =
    \lim_{d\rightarrow\infty} H(Z_d|Z^d)\leq \epsilon.\]
  But, $\epsilon$ was chosen arbitrarily, so
  \[\lim_{d\rightarrow\infty}\frac{1}{d}H(Z) = 0,\]
  a contradiction.
  Thus, $\lim_{d\rightarrow\infty}P_Z(r)=0$ for all $z$ in the domain
  of $Z$.

  Again, by contradiction, assume that
  $P_Z(\rho_k) \xrightarrow{P} 1$ as $d\rightarrow\infty$.
  Then
  \begin{align*}
    \sum_{d=1}^\infty \log P_{Z_d|Z^d}(\rho_k)
    &=\log\paren{\prod_{l=d}^\infty P_{Z_d|Z^d}(\rho_k)}\\
    &=\log\paren{P_Z(\rho_k)} \xrightarrow{P} 0
  \end{align*}
  using the continuous mapping theorem.
  But, the sum of non-positive values can converge to zero only
  if $\log P_{Z_d|Z^d}(\rho_k)=0$ for each $d$ with probability one. 
  Then $P_{Z}(\rho_k) = 1$ for each finite $d$.
  
  Fix $d$.
  For $P_{Z}(\rho_k) = P_Z(B(z,\rho_k)) = 1$, $B(z,\rho_k)$ must
  include the support of $Z$.
  Then $k$NN (in the $XYZ$ space) must be on a boundary of the
  domain of $Z$ and $\rho_k$, the $\ell_\infty$, $k$NN distance in
  $XYZ$, must be at least half of diameter of the domain of $Z$ with
  probability one.
  Because all observations are independent of each other and
  identically distributed, all $Z$-coordinates within the sample must
  also be on the boundary of the domain of $Z$ with probability one.
  If $Z$ were continuous, then the boundary would have measure zero,
  indicating that each coordinate of $Z$ must be discrete.
  Note that $Z$ coordinates need not be binary if using a
  discrete scalar distance metric for non-numeric, categorical variables.
  If the support of $Z$ contains more than one point, then ties are
  possible with positive probability, and $\rho_k=0$ with positive
  probability and $P_Z(B(z,\rho_k))<1$.
  Then $Z$ must have support on one point, again contradicting a
  non-zero entropy rate for $Z$.
  This indicates that $\lim_{d\rightarrow\infty}P_Z(\rho_k) < 1$.
  
  Using this fact, there must be an $r$ such that for each $d\geq 1$, 
  $P_Z(\rho_k)\leq P_Z(r)<1$, so that $P_Z(\rho_k)\leq
  P_Z(r)\rightarrow 0$ as $d\rightarrow\infty$.

  Finally, because $P_Z(\rho_k)\geq P_{XYZ}(\rho_k)$, we must have
  \[\frac{P_Z(\rho_k) -
      P_{XYZ}(\rho_k)}{1-P_{XYZ}(\rho_k)}\xrightarrow{P} 0\]
  as $d\rightarrow\infty$.
  Recall that $n_{Z}-k$ has a binomial distribution with the
  probability parameter stated above which converges to zero.
  From here, it is easy to see that
  $n_Z\xrightarrow{D} k$ (converges in distribution) as
  $d\rightarrow\infty$.
  Because $n_Z$ is converging to a constant, we also have
  $n_Z\xrightarrow{P} k$.
  But, $k\leq \tilde k, n_{XZ},n_{YZ}\leq n_Z$, so $\tilde k,
  n_{XZ},n_{YZ}\xrightarrow{P} k$ as well.
  By the continuous mapping theorem, for each sample point,
  \[\xi_i = \psi(k) - \psi(n_{XZ}) -
    \psi(n_{YZ})+\psi(n_Z)\xrightarrow{P} 0\]
  so that
  \[\hat I_{\text{prop}}(X;Y|Z) =
    \frac{1}{n}\sum_{i=1}^n\xi_i\xrightarrow{P} 0.\] 
\end{proof}

\section{Auxiliary Lemmas}\label{lemmas}

\begin{prop}\label{discFin}
Let $(X,2^X,\mu)$ be a discrete measure space with $\mu(X) = C<\infty$.
Then for every $\epsilon > 0$, there exists a finite set $E$ such that
$\mu(X\backslash E) < \epsilon$ and each point in $E$ has non-zero
measure.
\end{prop}
\begin{proof}
  If $X$ is finite, the problem is trivial.
  Assume $X$ in infinite.
  Without loss of generality, remove any zero-measure points from $X$.
  Because $(X,2^X,\mu)$ is discrete, $X$ must be countable so we number each
  point in $X$.
  We must have that $\sum_{i=1}^\infty \mu(x_i) = C$.
  Then there must be a positive integer, $N$, such that for each
  $n\geq N$, $C - \sum_{i=1}^n \mu(x_i) < \epsilon$.
  Let $E = \curly{x_i: 1\leq i\leq N}$.
  Then $\mu(X\backslash E) = \mu(X) - \mu(E) = C - \sum_{i=1}^N
  \mu(x_i)<\epsilon$.
\end{proof}

\begin{prop}\label{invBinom}
  Assume $W_n\sim\text{Binomial}(n,p)$, then
  \begin{equation}\E\brac{\frac{1}{W_n+1}} =
    \frac{1-(1-p)^{n+1}}{(n+1)p}\leq\frac{1}{np}
  \end{equation}
\end{prop}
\begin{proof}
  \begin{align*}
    \E\brac{\frac{1}{W_n+1}}
    &= \sum_{m=0}^n\frac{1}{m+1}\binom{n}{m}p^m(1-p)^{m-n}\\
    &= \frac{1}{(n+1)p}\sum_{m=0}^n\binom{n+1}{m+1}p^{m+1}(1-p)^{n-m}\\
    &=\frac{1}{(n+1)p}\sum_{m=1}^n\binom{n+1}{m}p^{m}(1-p)^{n+1-m}\\
    &= \frac{1}{(n+1)p} \brac{1 - \P(X_{n+1} = 0)}\\
    &= \frac{1 - (1-p)^{n+1}}{(n+1)p}
  \end{align*}
\end{proof}

\begin{prop}\label{absExpLog}
  Let $W\sim\text{Binomial}(n,p)$ then
  \begin{equation}
    \abs{\E\brac{\log\paren{\frac{W+k}{np+k}}}}\leq\frac{1}{np+k}
  \end{equation}
\end{prop}
\begin{proof}
  Using Taylor's theorem to expanding $\log(x)$ about $np+k$, there
  exists   $c\in[x,np+k]$ such that
  \begin{equation}\log(x) = \log(np+k) +\frac{x-np-k}{np+k} -
    \frac{(x-np-k)^2}{2c^2}.\end{equation}
  Plugging in $W+k$ for $x$ and aggregating the $\log$ terms,
  \begin{equation}\log\paren{\frac{W+k}{np+k}} = \frac{W - np}{np+k} -
    \frac{(W-np)^2}{2c^2}.\end{equation}
  Taking the expected value of both sides, the first-order term drops out,
  \begin{align*}
    \E\brac{\log\paren{\frac{W+k}{np+k}}}
      &= \E\brac{-\frac{(W - np)^2}{2c^2}}
  \end{align*}
  for some $c\in [np+k,W+k]$.
  Notice that $\E\brac{\log\paren{\frac{W+k}{np+k}}}\leq 0$ for all
  $c$, so that
  \begin{equation}\begin{split}
      \abs{\E\brac{\log\paren{\frac{W+k}{np+k}}}}
      \leq \E\brac{\max_{c\in[np+k,W+k]}
      \curly{\frac{(W - np)^2}{2c^2}}}.
  \end{split}\end{equation}
  Because $\frac{1}{2c^2}$ is monotonic, 
  $\frac{(W - np)^2}{2c^2}$ is optimized at the boundary values
  of $c=np+k$ and $c=W+k$.
  If $c = np+k$,
  \begin{align*}
    \E\brac{\frac{(W - np)^2}{2c^2}}
    &= \frac{np(1-p)}{2(np+k)^2}\\
    &\leq \frac{np+k}{2(np+k)^2}\\
    &= \frac{1}{2(np+k)}
  \end{align*}
  using $\E[(W-np)^2] = \text{Var}(W) = np(1-p)$.
  
  If $c=W+k$ and $k\leq np$, we use $\sum_{j=0}^n\binom{n+2}{j+2}p^{j+2}(1-p)^{n-j}
  = \P(V \geq 2)$ where $V\sim\text{Binomial}(n+2,p)$, so that
  \begin{align*}
    &\E\brac{\frac{(W - np)^2}{2(W+k)^2}}\\
    &= \frac{1}{2}\sum_{j=0}^n\frac{(j-np)^2}{(j+k)^2}
      \binom{n}{j}p^j(1-p)^{n-j}\\
    &\leq \frac{1}{2}\sum_{j=0}^n\frac{(j-np)^2}{(j+2)(j+1)}
      \binom{n}{j}p^j(1-p)^{n-j}\\
    &= \frac{1}{2}\sum_{j=0}^n\frac{(j-np)^2}{(n+2)(n+1)p^2}
      \binom{n+2}{j+2}p^{j+2}(1-p)^{n-j}\\
    &\leq \frac{\E\brac{(V-np)^2}}{2(n+2)(n+1)p^2}\\
    &= \frac{(n+2)p(1-p) + 4p^2}{2(n+2)(n+1)p^2}\\
    &\leq \frac{(n+2)p}{2(n+2)(n+1)p^2}\\
    &\leq \frac{1}{2np}\leq \frac{1}{np+k}
  \end{align*}
  for $n\geq4, k\geq2$ using $k\leq np$ in the last step.

  If $c = W+k\geq k$ and $np\leq k$,
  so 
  \begin{align*}
    \E\brac{\frac{(W - np)^2}{2c^2}}
    &\leq\E\brac{\frac{(W - np)^2}{2k^2}}\\
    &= \frac{np(1-p)}{2k^2}\\
    &\leq \frac{1}{2k}\leq \frac{1}{np+k}
  \end{align*}
  Putting this together,
  \begin{equation}
    \begin{split}
      &\abs{\E\brac{\log\paren{\frac{W+k}{np+k}}}}\\
      &\leq \max\curly{\abs{\E\brac{\frac{(W - np)^2}{2c^2}}},
        \abs{\E\brac{\frac{(W - np)^2}{2(W+k)^2}}}}\\
      &= \max\curly{\frac{1}{2(np+k)},\frac{1}{np+k}} = \frac{1}{np+k}
    \end{split}
  \end{equation}
\end{proof}

\begin{lem}\label{discBound}
  Assume $W_n-k\sim\text{Binomial}(n-k-1,p)$ and $k\geq\frac{p}{1-p}$.
  Then
  \begin{equation}
    \abs{\E\brac{\log(W_n)} - \log(np)}\leq \frac{1}{k} + \frac{k}{np}
  \end{equation}
  and
  \begin{equation}
    \abs{\E\brac{\psi(W_n)} - \log(np)}\leq \frac{2}{k} + \frac{k}{np}
  \end{equation}
\end{lem}
\begin{proof}
  Using the triangle inequality,
  \begin{align*}
    &\abs{\E\brac{\psi(W_n)} - \log(np)}\\
    &\leq \E\brac{\abs{\psi(W_n) - \log W_n}} +
      \abs{\E\brac{\log(W_n)} - \log(np)}.
  \end{align*}
  Using lemma \ref{absExpLog}, and the fact that $|\log(w)|\leq w-1$
  for $w>1$,
  we have that 
  \begin{align*}
    &\abs{\E\brac{\log(W_n)} - \log(np)}\\
    &\leq \abs{\E\brac{\log(W_n)} -\log((n-k-1)p+k)}\\
    &\quad + \abs{\log((n-k-1)p+k)- \log(np)}\\
    &= \abs{\E\brac{\log\paren{\frac{W_n}{(n-k-1)p+k}}}}\\
    &\quad + \log\paren{\frac{(n-k-1)p+k}{np}}\\
    &\leq \frac{1}{(n-k-1)p + k} + \frac{(n-k-1)p+k}{np}-1\\
    &\leq \frac{1}{k} + \frac{k}{np}.
  \end{align*}
  Because $k\geq\frac{p}{1-p}$, $(n-k-1)p+k\geq np$.

  Because $\abs{\psi(w)-\log(w)}<\frac{1}{w}$ for $w>0$ and $W_n\geq k$,
  $\E\brac{\abs{\psi(W_n) - \log W_n}} <\E\brac{\frac{1}{W_n}}\leq
  \frac{1}{k}$.
  So,
  \[\abs{\E\brac{\psi(W_n)} - \log(np)}\leq \frac{2}{k} + \frac{k}{np}.\]
\end{proof}

\begin{lem}\label{ktildelim}
  Let $V$ be a $d$-dimensional random variable on the probability
  space $(\mathcal{V},\mathcal{B}_{\mathcal{V}},P)$ with
  $\calV=\prod_{i\in I}\calV\subseteq\R^d$ where
  $I=\curly{1,\dots,d}$ and for nonempty $J\subseteq I$, let $P_J =
  P_{V_i:i\in J}$.
  Assume that the support of $P$ is $\calV$ and
  that for any nonempty $J\subseteq I$, the set
  \[D_J = \curly{w\in \prod_{i\in J}\calV_i: P_J(\{w\})>0}\]
  is countable and nowhere dense in $\prod_{i\in J}\calV_i$.
  Let $v_1,\dots,v_n\sim P$ be an independent sample in $\calV$, and
  for a point $v\in\calV$, define
  $\tilde k(v) = \abs{\curly{v_i:\norm{v-v_i}_\infty\leq \rho_v}}$
  where $\rho_v$ is the distance to the $k$th nearest neighbor to
  $v$ in the sample.
  If $\frac{k}{n}\rightarrow 0$, and $n\rightarrow\infty$ then
  \[\tilde k(V)\rightarrow k\text{ almost surely}\]
  given that
  $V\in C \equiv \curly{v\in \calV:P(\{v\}) = 0}$
\end{lem}
\begin{proof}
  For each $J\subseteq I$, $D_J$ is countable, so we can index it
  with the positive integers.
  Using contradiction, assume that some ordering is a Cauchy sequence;
  that is, for every $\epsilon>0$, there is a positive integer $N$ such
  that for all integers $l,m\geq N$, $\norm{a_l - a_m}_\infty < \epsilon$.
  But, all Cauchy sequences converge in the complete metric space
  (\cite[Theorem 3.11]{rudin1987real}),
  $(\R^d, \ell_\infty)$, so for some $a$, $a_i\rightarrow a$ as
  $i\rightarrow\infty$, a contradiction since $D$ is nowhere dense in
  $W$.
  Thus, for each $J$, there is a $\zeta_J>0$ such that for any two points,
  $a_l,a_m\in D_J, \norm{a_l - a_m}_\infty \geq\zeta_J$.
  $I$ is finite, so
  $\zeta \equiv \frac{\min_{J\subseteq I}\curly{\zeta_J}}{3}$
  exists.
  
  In $(\R^d,\ell_\infty)$, if $\norm{a-b}_\infty = \norm{a-c}_\infty$,
  then there must be at least one coordinate, $i$, such that
  $a(i)-b(i) = a(i)-c(i) \equiv r$ (where the vectors are function
  mapping the coordinate(s) to its coordinate value(s))
  and for all other coordinates, $j\neq i$, $a(j)-b(j), a(j)-c(j) \leq r$.
  This can only happen when $a(i),b(i),c(i)\in D_i$; they have a positive
  point mass so ties are possible.
  Consider a case where there are discrete points, $P_i(\{w(i)\})>0$
  for some coordinates, $i\in J$, but will not have any ties in distance.
  Suppose $A = \curly{(a(i):i\in I)}$ is a subset in the support of
  $P$ with positive measure
  such that the marginal distribution on $A$ is discrete for the
  coordinates in $J$ and continuous for coordinate in $I\backslash J$;
  that is, $P_i(\{a(i)\})>0$ when $i\in J$ and $P_i(\{a(i)\})=0$ when
  $i\in I\backslash J$. 
  Assume that for some point $(b(i): i\in J)$ in a subspace of A,
  $P_J((b_i: i\in J))>0$,
  then the subset of $A$ restricted to equal $(b(i): i\in J)$ on $J$,
  \[B\equiv \curly{(a(i):i\in I)\in A:a(i)=b(j), j\in J},\]
  also has a positive probability.
  If the random sample has values $v_m,v_l\in B$ and another arbitrary
  point $b\in B$, then
  \[\P\paren{\norm{v_m-b}_\infty = \norm{v_l-b}_\infty} = 0.\]
  This is because the scalar values of $v_m(i), v_l(i)$ and $b(i)$ are
  equal for $i\in J$ while for $i\in I\backslash J$, $v_m(i), v_l(i)$
  and $b(i)$ are from a continuous distribution, so equal with
  probability zero with each positive scalar distances.
  Further, if there are at least $k$ sample points in $B$, each will
  have a distinct $\ell_\infty$-distance to $b$ for the same reason.
  Thus, $\tilde k(b) = k$ with probability one.

  Generalizing on this point,
  let $v\in C$ and let $J = \curly{i\in I: P_i(\{v(i)\}) = 0}$.
  Define $B_J(v,\delta)\subseteq \calV$ to be the Cartesian product of
  $[v(i)-\delta,v(i)+\delta]$ for $i\in J$ and $\{v(i)\}$ for $i\in
  I\backslash J$,
  \[B_J(v,\delta) = \prod_{i\in J}[v(i)-\delta,v(i)+\delta]\times
    \prod_{i\in I\backslash J}\{v(i)\}.\]
  $\calV_J$ may have positive point masses among continuous points.
  Because $D_J$ is nowhere dense in $\calV_J$, there is $\delta_v > 0$
  such that $D_J\cap \pi_J(B_J(v,\delta_v)) = \varnothing$ 
  (where $\pi_J$ is the projection onto $J$) and
  $P_J(B_J(v,\delta_v))>0$.
  Notice that if there are more than $k$ sample points in
  $B(v,\delta_v)$ then $\tilde k(v) = k$.

  Let $W = \curly{v\in C: \delta_v \geq \zeta}$ and fix $w\in W$.
  Let $\epsilon\in [\zeta, 0)$ and $\rho_v$ be the
  $\ell_\infty$-distance from $v$ to its $k$NN in the sample
  $v_1,\dots, v_n$. 
  Choose $N$ large enough so that for all $n\geq N$,
  $\frac{k}{n}\leq P(B_J(v,\epsilon))$.
  Then using Chernoff's bound,
  \begin{align*}
    \P(\rho_v > \epsilon)
    &=\P(\text{Binomial}(n,P(v,\epsilon)) \leq k)\\
    &\leq \exp\curly{-\paren{\frac{1}{2}nP(v,\epsilon)-k}}.
  \end{align*}
  So, $\sum_{n=1}^\infty\P(\rho_v>\epsilon)<\infty$.
  Using the Borel-Cantelli lemma, \cite[Lemma
  2.2.4]{dembo2019probability}, $\rho_v\rightarrow 0$ almost surely
  as $n\rightarrow\infty$.

  Notice that
  \begin{equation*}
    \P(\rho_V > \epsilon|V\in W)
    =\int_W \P(\rho_v > \epsilon)dP(v).
  \end{equation*}
  Using the Lebesgue dominated convergence theorem, with the fact that
  for each $n, \P(\rho_v > \epsilon)\leq 1$ for all $v\in W$ and
  $\P(\rho_v > \epsilon)\rightarrow 0$ almost surely, we have
  $\P(\rho_V > \epsilon|V\in W)\rightarrow 0$ almost surely
  as $n\rightarrow\infty$.
  Then $\tilde k(V)\rightarrow k$ given that
  $V\in W$ almost surely as $n\rightarrow\infty$.
  
  Consider
  \[C\backslash W = \curly{v\in C: \delta_v< \zeta}.\]
  For each $v\in C\backslash W$, there must be $J\subseteq I$ such
  that 
  \[D\equiv (D_J\times \curly{v(I\backslash J)}) \cap
    B_J(v,\zeta)\neq\varnothing.\]
  There may be points $x\in D$ such that $P(\{x\}) = 0$. 
  Notice that
  \[C\backslash W = \bigcup_{x\in D} B(x,\zeta).\]
  Similarly, for each $x\in D$, there is $J\subseteq I$ such that
  $x(J)\in D_J$.
  Because $D\subseteq \bigcup_{J\subseteq I}(D_J\times D_{I\backslash
    J})$, $D$ is countable.
  By choice of $\zeta$, for every two points $a,b\in D$,
  $\norm{a-b}_\infty>\zeta$, so
  $i\neq j, B(x_i,\zeta)\cap B(x_j,\zeta) = \varnothing$..
  With both of these,
  \[P\paren{\bigcup_{x\in D} B(x,\zeta)} =
    \sum_{i=1}^\infty P\paren{B(x_i,\zeta)}.\]

  For $x_i\in D$, for all $v\in B(x_i,\zeta)$, there is no $J$, such
  that $v(J)\in D_J$ because of choice of $\zeta$.
  Stated differently, for each $J\subseteq I$, $P(\{v(J)\}) = 0$.
  Consequently, there can be no ties in distance to points other than
  $x_i$.
  Let $ K_{x_i}= \curly{v\in B(x_i,\zeta)\backslash \{x_i\}:
    x_i\in B(v,\rho_v)}$
  Using \cite[Corollary 6.1]{gyorfi2006distributionfree},
  $\abs{K_{x_i}}\leq k\gamma_d$
  where $\gamma_d$ is a function of only the dimension $d$.
  Let $p_i = P(B(x_i,\zeta)\backslash\{x_i\})$, then
  \begin{align*}
    &\P\paren{\tilde k(v) > k: v\in B(x_i,\zeta)\backslash\{x_i\}}\\
    & \leq \P(x_i\in B(v,\rho_v))\\
    & = \P(v\in K_{x_i}).
  \end{align*}
  This probability depends on the number of sample points that fall
  into $B(x_i,\zeta)$.
  Looking at the random variable and using Chernoff,
  \begin{align*}
    &\P(\tilde k(V) > k| V\in B(x_i,\zeta)\backslash\{x_i\})\\
    &\leq\P(V\in K_{x_i}|V\in B(x_i,\zeta)\backslash \{x_i\})\\
    &=\P\paren{\text{Binomial}(n, p_i)\leq k\gamma_d}\\
    &\leq \exp\curly{-\paren{\frac{1}{2}np_i-k\gamma_d}}.
  \end{align*}
  So, $\sum_{n=1}^\infty\P(\tilde k(V) > k| V\in
  B(x_i,\zeta)\backslash\{x_i\})<\infty$.
  Using the Borel-Cantelli lemma, \cite[Lemma
  2.2.4]{dembo2019probability}, $\tilde k(V)\rightarrow k$ given that
  $V\in C\backslash W$ almost surely as $n\rightarrow\infty$.
\end{proof}

\begin{lem}\label{distDeriv}
  Let $F_\rho(r)$ be the probability that the distance to a point's
  $k$NN in a sample of $n$ points is $\rho \leq r$ and let $P_W(r)$ be
  the probability mass of the ball of radius $r$ centered at the same
  point.
  Then
  \begin{equation}\begin{split}
      \frac{dF_\rho}{dP_W}(r)
      &= \frac{(n-1)!}{(k-1)!(n-k-1)!}\times\\
      &\qquad\brac{P_W(r)}^{k-1}\brac{1-P_W(r)}^{n-k-1}.
  \end{split}\end{equation}
\end{lem}
\begin{proof}
  Let $\rho_1,\hdots,\rho_{n-1}$ be the ordered distances from the
  point of interest.
  The probability that $k$th largest distance is at least $r$ is
  \begin{align*}
    &\P\paren{\rho_k\leq r}\\
    &= \P\paren{I(\rho_i\leq r)\geq k}\\
    &= \sum_{j=k}^{n-1} \P\paren{I(\rho_i\leq r) = j}\\
    &= \sum_{j=k}^{n-1} \binom{n-1}{j}
      \brac{P_W(r)}^{j}\brac{1-P_W(r)}^{n-j-1}.
  \end{align*}
  Taking the derivative with respect to $P_w(r)\equiv p$,
  \begin{align*}
    &\frac{dF_\rho}{dp}\\
    &= \sum_{j=k}^{n-1} \binom{n-1}{j}\frac{d}{dp}
      \brac{p^{j}(1-p)^{n-j-1}}\\
    &= \sum_{j=k}^{n-1} \binom{n-1}{j}
      [jp^{j-1}(1-p)^{n-j-1}\\
    &\quad - p^{j}(n-j-1)(1-p)^{n-j-2}]\\
    &= \sum_{j=k}^{n-1} \frac{(n-1)!}{(j-1)!(n-j-1)}
      p^{j-1}(1-p)^{n-j-1}\\
    &\quad - \sum_{j=k}^{n-1} \frac{(n-1)!}{j!(n-j-2)}
      p^{j}(1-p)^{n-j-2}\\
    &= \frac{(n-1)!}{(k-1)!(n-k-1)!}p^{k-1}(1-p)^{n-k-1}.
  \end{align*}
  The last equality follows from realizing that all terms cancel
  except for $j=k$ in the first term.
\end{proof}


%

\begin{defn}\label{nonsing}
  Let $(W,\mathcal{B},P)$ be a $d$-dimensional probability space with
  $W = \prod_{i\in I}W_i$ where $I = \curly{1,\dots,d}$ and
  $P_J = P_{W_i:i\in J}$ for $J\subseteq I$.
  For $A\subseteq W$, and
  $v= (v_i:i\in J)\in\prod_{i\in J}W_i$ let
  $A_{v} = \curly{\paren{a_i:i\in I\backslash J}: (a_i:i\in I)\in
    A, a_i=v_j, i=j\in J}$.
  The probability measure, $P$, is \emph{non-singular} if for some
  $J\subseteq I$ and $A\subseteq W$ in the support of $P$,
  \begin{equation}
    \begin{split}
      P&\left(\left\{(a_i:i\in I):\right.\right.\\
          &\left.\left. P_{I\backslash J}(A_{(a_i:i\in J)}) =
        P_J(A_{(a_i:i\in I\backslash J)}) = 0\right\}\right)=0.
    \end{split}
  \end{equation}
\end{defn}

\begin{lem}\label{domMeasure}
  Let $V = (V_1, V_2,\dots, V_d)$ be a $d$-dimensional random vector
  on the probability space,
  $(\mathcal{V},\mathcal{B}_{\mathcal{V}},P_V)$ 
  where $V_i$ is either continuous, countably discrete, or a mix of
  both.
  If $P$ is non-singular then there exists a product measure $\mu$ on
  the same space such that $P_V\ll\mu$.\\
\end{lem}
\begin{proof}
  We construct $\mu$ by looking at the scalar coordinates of
  $V\equiv (V_1,\dots,V_d)$ over its product space,
  $\calV\equiv \calV_1\times\calV_2\times\dots\times\calV_d$.
  If $\calV_i$ is not a subset of $\R$, $V_i$ is categorical and we
  use a zero-one distance metric.
  So that we can work exclusively in $\R^c$ for some positive integer
  $c$, we create dummy indicators for all categories except one; this
  preserves the $\ell_\infty$ metric for categorical variables.
  Recall that the marginal measure for any scalar coordinate is
  $P_{V_i}(A) = P_V(\calV_1\dots\times\calV_{i-1}\times
  A\times \calV_{i+1}\times\dots\times\calV_d)$ where
  $A\subseteq\calV_i$.
  For each $i = 1,\dots,d$, redefine $\calV_i$ by restricting it to
  the support of $P_{V_i}$ and $\mathcal{B}_{\calV_i}$ the
  corresponding $\sigma$-algebra.
  Partition $\calV_i$ into its discrete and continuous parts.
  For a set $A$ contained within the support of a random variable, $U$,
  let $C_U(A) = \curly{x\in A: P_U(x)=0}$ be the continuous partition and
  $D_U(A) = \curly{x\in A: P_U(x)>0}$, which is countable by
  assumption.
  Clearly $C_U(A)\cup D_U(A) = A$ and $C_U(A)\cap D_U(A) =
  \varnothing$ for all random variables $U$. 
  Let $\lambda$ be the Lebesgue measure and $\nu$ be the counting
  measure. 
  Define the measure
  $\mu_i:\mathcal{B}_{\calV_i}\rightarrow [0,\infty)$ to be 
  $\lambda + \nu_i$, where $\nu_i(C_{V_i}(\calV_i))=0$ and the
  counting measure on $D_{V_i}(\calV_i)$,
  $\nu_i(D_{V_i}(\calV_i)) = \nu(D_{V_i}(\calV_i))$.
  It is easy to see that $\mu_i$ is a well-defined measure on the
  measurable space, $(\calV_i,\mathcal{B}_{\calV_i})$ because both the 
  counting measure and Lebesgue measures are well-defined, as is their
  sum.
  Define a measure $\mu:\mathcal{B}_{\calV}\rightarrow\R$ as the
  product measure, $\mu = \mu_1\times\mu_2\times\dots\times\mu_d$.

  With the construction complete, we now show that $P_V\ll\mu$.
  We begin by showing that for each coordinate, $j=1,\dots,d$,
  $P_{\calV_j}\ll\mu_j$.
  Let $j = 1,\dots,d$ and $A\in\mathcal{B}_{\calV_i}$ with
  $\mu_j(A)=0$.
  Consider the continuous and discrete partitions, $C_{V_j}(A)$ and
  $D_{V_j}(A)$, respectively.
  By definition, $\lambda(C_{V_j}(A)) + \nu_j(D_{V_j}(A)) = 0$ so  
  $\lambda(C_{V_j}(A)) = 0$ and $\nu_j(D_{V_j}(A)) = 0$.
  If the coordinate project for $j$ has a nonempty continuous
  partition, then $P_{V_j}\ll \lambda$ on $C_{V_j}(\calV_j)$, so
  $P_{V_j}(C_{V_j}(A)) = 0$.
  Also, $0 = \nu_j(D_{V_j}(A)) = \nu(D_{V_j}(A))$, so $D_{V_j}(A) =
  \varnothing$, so $P_{V_j}(D_{V_J}(A)) = 0$.
  Then $P_{V_j}(A) = P_{V_j}(C_{V_j}(A)) + P_{V_j}(D_{V_j}(A)) = 0$

  Proceeding by mathematical induction, we already have $P_{V_1}\ll
  \mu_1$.
  Assume that $P_{V_1\dots V_j}\ll \mu_1\times\dots\times\mu_j\equiv
  \prod_{i=1}^j\mu_i$ and 
  that for some   $A\in\mathcal{B}_{\calV_1\dots\calV_j\calV_{j+1}}$
  (the product   $\sigma$-algebra) $(\prod_{i=1}^{j+1}\mu_i)(A) = 0$.
  Let $A_{v_1,\dots,v_j} = \curly{v_{j+1}:(v_1,\dots,v_j,v_{j+1})\in
    A}$ and 
  $A_{v_{j+1}}= \curly{(v_1,\dots,v_j): (v_1,\dots,v_j,v_{j+1})\in
    A}$.
  Let $A_1 = \calV_1\times\dots\times\calV_j
  \curly{v_{j+1}: P_{V_1\dots V_j}(A_{v_{j+1}})>0}$ and
  $A_2 = \curly{(v_1,\dots,v_j):
    P_{V_{j+1}}(A_{v_1,\dots,v_j})>0}\times \calV_{j+1}$.
  
  Using Fubini's theorem,
  \begin{align*}
    0 &= \paren{\prod_{i=1}^{j+1}\mu_i}(A)\\
      &= \paren{\prod_{i=1}^{j}\mu_i\times \mu_{j+1}}(A)\\
      &= \int_{\calV_{j+1}} \paren{\prod_{i=1}^{j}\mu_i}(A_{v_{j+1}})
        d\mu_{j+1}(v_{j+1}).
  \end{align*}
  
  Using \cite[Lemma 1.3.8]{dembo2019probability}, $f\geq0, \int
  fd\mu=0\Rightarrow \mu\curly{x:f(x)>0}=0$, we must have
  \begin{align*}
    0 &= \mu_{j+1}\paren{\curly{v_{j+1}:
        \paren{\prod_{i=1}^{j}\mu_i}A_{v_{j+1}})>0}}\\
      &= \mu_{j+1}\paren{\calV_{j+1}\backslash\curly{v_{j+1}: 
        \paren{\prod_{i=1}^{j}\mu_i}(A_{v_{j+1}})=0}}\\
      &\geq \mu_{j+1}\paren{\calV_{j+1}\backslash\curly{v_{j+1}: 
        P_{V_1\dots V_j}(A_{v_{j+1}})=0}}.
  \end{align*}
  The last inequality follows because $P_{V_1\dots V_j}\ll
  \prod_{i=1}^{j}\mu_i$ implies that
  $\curly{v_{j+1}:\paren{\prod_{i=1}^{j}\mu_i}(A_{v_{j+1}})=0}
  \subseteq \curly{v_{j+1}: P_{V_1\dots V_j}(A_{v_{j+1}})=0}$.
  Then $\mu_{j+1}\paren{\calV_{j+1}\backslash\curly{v_{j+1}: 
      P_{V_1\dots V_j}(A_{v_{j+1}})=0}} = 0$.
  But, $P_{V_{j+1}} \ll \mu_{j+1}$ implies that
  \begin{align*}
    0 &= P_{V_{j+1}}\paren{\calV_{j+1}\backslash\curly{v_{j+1}: 
        P_{V_1\dots V_j}(A_{v_{j+1}})=0}}\\
      &= P_{V_{j+1}}\paren{\curly{v_{j+1}:
        P_{V_1\dots V_j}(A_{v_{j+1}})>0}}\\
      &= P_{V_1\dots V_jV_{j+1}}(A_1).
  \end{align*}

  Using the same procedure but switching $\prod_{i=1}^j\mu_i$ and
  $\mu_{j+1}$ and correspondingly, switching $P_{V_1\dots V_j}$ and
  $P_{V_{j+1}}$, it is easy to show that
  \begin{align*}
  0 &= P_{V_1\dots V_j}\paren{\curly{(v_1,\dots,v_j):
      P_{V_{j+1}}(A_{v_1,\dots,v_j})>0}}\\
    &= P_{V_1\dots V_jV_{j+1}}(A_2).
  \end{align*}
  
  Consider the set of points $(v_1,\dots,v_j,v_{j+1})$ such that each
  coordinate satisfies $P_{V_{j+1}}(A_{v_1,\dots,v_j}) = 0$ and $P_{V_1\dots
    V_j}(A_{v_{j+1}}) = 0$; call this set, $A_3$.
  Showing that $P(A_3)=0$, consider the set of points, $(a_1,\dots
  a_d)\in B\subseteq\calV$ such that 
  \[P_{V_{j+1}\dots V_d}\paren{\brac{A\times
        \prod_{i=j+2}^d\calV_i}_{(a_1,\dots,a_j)}}=0\]
  and \[P_{V_1\dots V_j}\paren{\brac{A\times
        \prod_{i=j+2}^d\calV_i}_{(a_{j+1},\dots,a_d)}}=0.\]
  Let $(b_1,\dots,b_d)\in A_3\times\prod_{i=j+2}^d\calV_i$.
  Then
  \begin{align*}
    &P_{V_{j+1}\dots V_d}\paren{\brac{A\times
      \prod_{i=j+2}^d\calV_i}_{(b_1,\dots,b_j)}}\\
    &= P_{V_{j+1}\dots V_d}\paren{A_{(b_1,\dots,b_j)}\times
      \prod_{i=j+2}^d\calV_i}\\
    &= P_{V_{j+1}}\paren{A_{(b_1,\dots,b_j)}} = 0
  \end{align*}
  and
  \begin{align*}
    &P_{V_1\dots V_j}\paren{\brac{A\times
      \prod_{i=j+2}^d\calV_i}_{(b_{j+1},\dots,b_d)}}\\
    &= P_{V_1\dots V_j}\paren{A_{b_{j+1}}} = 0.
  \end{align*}
  Then
  \[A_3\times\prod_{i=j+2}^d\calV_i\subseteq B.\]
  Because $P_V$ is non-singular, $P_V(B) = 0$, so
  \[P_V\paren{A_3\times\prod_{i=j+2}^d\calV_i} = P_{V_1\dots
        V_jV_{j+1}}(A_3) = 0.\]
  
  Now, $A\subseteq A_1\cup A_2\cup A_3$ implies that
  \begin{align*}
    &P_{V_1\dots V_j V_{j+1}}(A)\\
    &\leq P_{V_1\dots V_j V_{j+1}}(A_1\cup A_2\cup A_3)\\
    &\leq P_{V_1\dots V_j V_{j+1}}(A_1) + P_{V_1\dots V_j V_{j+1}}(A_2)\\
    &\quad + P_{V_1\dots V_j V_{j+1}}(A_3)\\
    &= 0,
  \end{align*}
  so $P_{V_1\dots V_j V_{j+1}}(A) = 0$.
  Thus, by mathematical induction, for any positive integer, $d$, we
  have that $P_V\ll \mu$.
\end{proof}

\begin{lem}\label{rnDerivLim}
  Let $\mu$ and $\nu$ be nonsingular probability measures 
  on $(\R^d,\mathcal{B})$ such that $\nu\ll\mu$
  and assume $\curly{x:\mu(\{x\})>0}$ is nowhere dense in $\R^d$.
  Let $B(x,r)$ be a ball of radius $r$ centered at $x$.
  If $\mu(\{x\})>0$ then
  \[\frac{d\nu}{d\mu}(x) = \frac{\nu(\{x\})}{\mu(\{x\})}\]
  otherwise
  \begin{equation}
    \frac{d\nu}{d\mu}(x) = 
    \lim_{r\rightarrow 0}\frac{\nu(B(x,r))}{\mu(B(x,r))}.
  \end{equation}
\end{lem}
\begin{proof}
  If $\mu(\{x\})>0$,
  then $\frac{d\nu}{d\mu}(x) = \frac{\nu(\{x\})}{\mu(\{x\})}$:
  \[\int_{\{x\}}\frac{\nu(\{x\})}{\mu(\{x\})}d\mu
    = \frac{\nu(\{x\})}{\mu(\{x\})} \mu(\{x\}) = \nu(\{x\}).\]
  
  If $\mu(\{x\})=0$ and in the support of $\mu$, there must be some
  $\delta>0$ such that for every $y\in B(x,\delta), \mu(\{y\})=0$
  because $\curly{x:\mu(\{x\})>0}$ is nowhere dense in $\R^d$ and
  $\mu(B(x,\delta))>0$.
  Notice that some coordinates of $x = (x_1,\dots,x_d)$ may be
  discrete but there must be at least one continuous coordinate in
  order for $\mu(\{x\})=0$.
  Let $I_{\text{cont}}$ be the index of continuous coordinates of $x$
  and $I_{\text{disc}}$ be the index of discrete coordinates of $x$.
  Using the proof of lemma~\ref{domMeasure}, each coordinate of
  $I_{\text{cont}}$ will be dominated by the Lebesgue measure within
  $B(x,\delta)$.
  Again, because $\curly{x:\mu(\{x\})>0}$ is nowhere dense in $\R^d$,
  \[\delta_{\text{disc}} \equiv\min_{y\in\text{Supp}(\mu)}
    \curly{\norm{x_{\text{disc}} - y_{\text{disc}}}_{\ell_\infty}:
      x_{\text{disc}} \neq y_{\text{disc}}}>0\]
  where $z_{\text{disc}} \equiv (z_i:I_{\text{disc}})$ for $z=x,y$.
  If $\delta > \delta_{\text{disc}}$, then redefine $\delta =
  \delta_{\text{disc}}$.
  Now $x_{\text{disc}}$ is constant within $B(x,\delta)$ and
  homeomorphic to a subset of $\R^a$ for some integer $a\leq d$ with
  the corresponding Lebesgue measure.
  Then $\mu\ll\lambda$ on $B(x,\delta)$ where $\lambda$ is Lebesgue on
  the support of $\mu$ and zero otherwise so that $\lambda\ll\mu$ as
  well.
  Using \cite[Theorem 7.8]{rudin1987real},
  \[\frac{d\nu}{d\lambda}(x) =
    \lim_{r\rightarrow0}\frac{\nu(B(x,r))}{\lambda(B(x,r))}\]
  and
  \[\frac{d\mu}{d\lambda}(x) =
    \lim_{r\rightarrow0}\frac{\mu(B(x,r))}{\lambda(B(x,r))}.\]
  Notice that $\mu\ll\lambda$ and $\lambda\ll\mu\Rightarrow
  \frac{d\mu}{d\lambda}(x)>0$.
  Then
  \begin{align*}
    \frac{d\nu}{d\mu}(x)
    &= \paren{\frac{d\nu}{d\lambda}\frac{d\lambda}{d\mu}}(x)\\
    &= \brac{\frac{d\nu}{d\lambda}
      \paren{\frac{d\lambda}{d\mu}}^{-1}}(x)\\
    &= \paren{\lim_{r\rightarrow0}
      \frac{\nu(B(x,r))}{\lambda(B(x,r))}}
      \paren{\lim_{r\rightarrow0}
      \frac{\mu(B(x,r))}{\lambda(B(x,r))}}^{-1}\\
    &= \lim_{r\rightarrow0}\frac{\nu(B(x,r))}{\mu(B(x,r))}.
  \end{align*}
\end{proof}

\begin{lem}\label{cmiDeriv}
  Assume $0<f\leq C$, for some $C > 0$ if $P_{XYZ}(\{(x,y,z)\})>0$ then
  \[f(x,y,z) = \frac{P_{XYZ}(\{(x,y,z)\})P_Z(\{z\})}
    {P_{XZ}(\{(x,z)\})P_{YZ}(\{(y,z)\})}\]
  otherwise
  \begin{equation}\frac{P_{XYZ}(r)P_Z(r)}{P_{XZ}(r)P_{YZ}(r)}\rightarrow
    \frac{dP_{XY|Z}}{d(P_{X|Z}\times P_{Y|Z})}
  \end{equation}
  (converges pointwise) as $r\rightarrow0$ and
  \begin{equation}\frac{P_{XYZ}(r)P_Z(r)}{P_{XZ}(r)P_{YZ}(r)}\leq C
  \end{equation}
  almost everywhere $[P_{X|Z}\times P_{Y|Z}]$.
\end{lem}

\begin{proof}
  From lemma \ref{domMeasure}, for a random variables, $U$ and
  $V$ define $\mu_U\times\mu_V = \mu_{UV}$.
  Based on definitions from \cite[\S 7.1 and \S 7.2]{gray2011entropy}, define
  $p_{UV} = \frac{dP_{UV}}{d\mu_{UV}}$, $p_V = \frac{dP_V}{d\mu_V}$,
  and $p_{U|V} = \frac{p_{UV}}{p_{V}}$.
  Note that $\mu_{UV}$ is not a probability measure, but for brevity, we
  define $\mu_{U|V}(A|v) = \mu_U(A)$ for $A$ in the support of $U$ and
  $v$ in the support of $V$ so that $P_{U|V}$ and $\mu_{U|V}$ have the
  same support.

  From lemma \ref{domMeasure}, $P_{XY|Z}\ll\mu_{XY|Z}$.
  Because $P_{X|Z}$ has the same support as $\mu_{X|Z}$, $\mu_{X|Z}\ll
  P_{X|Z}$; similarly,
  $P_{Y|Z}$ has the same support as $\mu_{Y|Z}$, so
  $\mu_{Y|Z}\ll\P_{Y|Z}$.
  Using a proof similar to that of lemma \ref{domMeasure},
  $\mu_{X|Z}\times\mu_{Y|Z}\ll P_{X|Z}\times P_{Y|Z}$.
  But, $\mu_{XY|Z} = \mu_{X|Z}\times\mu_{Y|Z}$ because it is a product
  measure.
  
  Using properties of RN derivatives
  so that 
  \begin{align*}
    &\frac{dP_{XY|Z}}{d(P_{X|Z}\times P_{Y|Z})}\\
    &= \frac{dP_{XY|Z}}{d(\mu_{X|Z}\times \mu_{Y|Z})}
      \frac{d(\mu_{X|Z}\times \mu_{Y|Z})}{d(P_{X|Z}\times P_{Y|Z})}\\
    &= \frac{dP_{XY|Z}}{d\mu_{XY|Z}}
      \brac{\frac{d(P_{X|Z}\times P_{Y|Z})}
      {d(\mu_{X|Z}\times\mu_{Y|Z})}}^{-1}\\
    &= \frac{dP_{XY|Z}/d\mu_{XY|Z}}
      {d(P_{X|Z}\times P_{Y|Z})/d(\mu_{X|Z}\times \mu_{Y|Z})}\\
    &= \frac{dP_{XY|Z}/d\mu_{XY|Z}}
      {(dP_{X|Z}/d\mu_{X|Z})(dP_{Y|Z}/d\mu_{Y|Z})}\\
    &= \frac{\frac{dP_{XYZ}}{d\mu_{XYZ}}/\frac{dP_Z}{d\mu_Z}}
      {\paren{\frac{dP_{XZ}}{d\mu_{XZ}}/\frac{dP_Z}{d\mu_Z}}
      \paren{\frac{dP_{YZ}}{d\mu_{YZ}}/\frac{dP_Z}{d\mu_Z}}}\\
    &= \frac{\frac{dP_{XYZ}}{d\mu_{XYZ}}\frac{dP_Z}{d\mu_Z}}
      {\frac{dP_{XZ}}{d\mu_{XZ}}\frac{dP_{YZ}}{d\mu_{YZ}}}\\
    &= \frac{d(P_{XYZ}\times P_Z)/d(\mu_{XYZ}\times \mu_Z)}
      {d(P_{XZ}\times P_{YZ})/d(\mu_{XZ}\times\mu_{YZ})}\\
    &= \frac{d(P_{XYZ}\times P_Z)}{d(P_{XZ}\times P_{YZ})}.
  \end{align*}
  Applying lemma~\ref{rnDerivLim} completes the first claim.

  Second, note that $g\equiv\frac{d\nu}{d\mu}\leq C$
  implies that for any set $A$ such that
  $\mu(A)>0, \frac{\nu(A)}{\mu(A)} \leq C$.
  To see this, $\nu(A) = \int_A gd\mu \leq \int_A Cd\mu = C\mu(A)$.
  So, the second claim holds as well.
\end{proof}

\begin{lem}\label{contBound}
  Assume $W_{n,r}-k\sim \text{Binomial}\paren{n-k-1,
    \frac{q(r)-p(r)}{1-p(r)}}$ where
  $p(r), q(r)$ are probabilities, and for all $r$,
  $p(r) \leq q(r)$.
  Then
  \begin{equation}\begin{split}
      &\abs{\int_0^\infty\E\brac{\psi(W_{n,r})} -
        \log(nq(r))dF_\rho(r)}
      < \frac{3}{k-1}
    \end{split}\end{equation}
  and
  \begin{equation}\begin{split}
      &\abs{\int_0^\infty\E\brac{\log(W_{n,r})} -
        \log(nq(r))dF_\rho(r)}
      < \frac{2}{k-1}.
    \end{split}\end{equation}
\end{lem}
\begin{proof}
  We suppress the arguments/subscripts, $r$ and $n$ for brevity
  through out this proof. 
  Using the triangle inequality, the
  fact that $\abs{\psi(w) - \log(w)}\leq\frac{1}{w}$,
  \begin{align*}
    &\abs{\E\brac{\psi(W_{n,r})} - \log(nq(r))}\\
    &\equiv\abs{\E\brac{\psi(W)} - \log(nq)}\\
    &\leq \abs{\E\brac{\psi(W)} - \E\brac{\log(W_n)}}\\
    &\quad + \abs{\E\brac{\log(W)} - \log\brac{(n-k-1)
      \paren{\frac{q-p}{1-p}} + k}}\\
    &\quad + \abs{\log\brac{(n-k-1)\paren{\frac{q-p}{1-p}} + k} - \log(nq)}\\
    &\leq \E\brac{\frac{1}{W}}
      + \abs{\E\brac{\log\paren{\frac{W}{(n-k-1)
      \paren{\frac{q-p}{1-p}} + k}}}}\\
    &\quad + \log\paren{\frac{(n-k-1)
      \paren{\frac{q-p}{1-p}} + k}{nq}}\\
    &\leq \frac{1}{k} + \frac{1}{(n-k-1)\paren{\frac{q-p}{1-p}} + k}
      + \frac{k}{np}-1\\
    &\leq \frac{2}{k} + \frac{k}{np}-1.
  \end{align*}
  The penultimate step uses $W\geq k$ and proposition~\ref{absExpLog}
  for the first two terms.
  We show the third term here, again using $\log(w)\leq w-1$ for
  $w\geq0$ and $\paren{\frac{p(1-q)}{q(1-p)}}\leq 1$: 

  \begin{align*}
    &\log\paren{\frac{(n-k-1)\paren{\frac{q-p}{1-p}} + k}{nq}}\\
    &\leq \frac{(n-k-1)\paren{\frac{q-p}{1-p}} + k}{nq}-1\\
    &= \frac{k\paren{\frac{1-q}{1-p}} +
      (n-1)\paren{\frac{q-p}{1-p}}}{nq} -1\\
    &= \frac{k}{np}\paren{\frac{p(1-q)}{q(1-p)}}
      + \frac{n-1}{n}\paren{\frac{q-p}{q(1-p)}}-1\\
    &= \frac{k}{np}\paren{\frac{p(1-q)}{q(1-p)}}
      + \frac{n-1}{n}\paren{1- \frac{p(1-q)}{q(1-p)}}-1\\
    &= \paren{\frac{k}{np} - \frac{n-1}{n}}
      \paren{\frac{p(1-q)}{q(1-p)}} + \frac{1}{n}\\
    &\leq \frac{k}{np} - \frac{n-1}{n} + \frac{1}{n} = \frac{k}{np}-1.
  \end{align*}

  Putting this all together,
  \begin{align*}
    &\abs{\int_0^\infty\E\brac{\psi(W)}  - \log(nq)dF_\rho}\\
    &\leq \int_0^\infty\abs{\E\brac{\psi(W)} - \log(nq)}dF_\rho\\
    &\leq \int_0^\infty\paren{\frac{2}{k} + \frac{k}{np}-1}dF_\rho\\
    &= \frac{2}{k} + \frac{k}{n}\int_0^\infty\frac{1}{p}dF_\rho-1\\
    &= \frac{2}{k} + \frac{k}{n}\paren{\frac{n-1}{k-1}}-1\\
    &\leq \frac{2}{k} + \frac{k}{k-1}-1 \leq \frac{3}{k-1}
  \end{align*}

  We complete the integration step using lemma~\ref{distDeriv} and two
  beta function identities, 
  \begin{align*}
    &\int_0^\infty\frac{1}{p}dF_\rho\\
    &= \int_0^1\frac{1}{p}\frac{(n-1)!}{(k-1)!(n-k-1)!}
      p^{k-1}(1-p)^{n-k-1}dp\\
    &= \frac{(n-1)!}{(k-1)!(n-k-1)!}\int_0^1
      p^{k-2}(1-p)^{n-k-1}dp\\
    &= \frac{(n-1)!}{(k-1)!(n-k-1)!}\frac{(k-2)!(n-k-1)!}{(n-2)!}\\
    &= \frac{n-1}{k-1}.\\
  \end{align*}
  The second claim follows using a close but simpler argument.
\end{proof}

\ifCLASSOPTIONcompsoc
  \section*{Acknowledgments}
\else
  \section*{Acknowledgment}
\fi

The authors would like to thank the Ford Foundation Dissertation
Fellowship for funding this work.

\bibliographystyle{IEEEtran}
\bibliography{IEEEabrv,cmi.bib,locusts.bib}

\begin{thebibliography}{10}
\providecommand{\url}[1]{#1}
\csname url@samestyle\endcsname
\providecommand{\newblock}{\relax}
\providecommand{\bibinfo}[2]{#2}
\providecommand{\BIBentrySTDinterwordspacing}{\spaceskip=0pt\relax}
\providecommand{\BIBentryALTinterwordstretchfactor}{4}
\providecommand{\BIBentryALTinterwordspacing}{\spaceskip=\fontdimen2\font plus
\BIBentryALTinterwordstretchfactor\fontdimen3\font minus
  \fontdimen4\font\relax}
\providecommand{\BIBforeignlanguage}[2]{{%
\expandafter\ifx\csname l@#1\endcsname\relax
\typeout{** WARNING: IEEEtran.bst: No hyphenation pattern has been}%
\typeout{** loaded for the language `#1'. Using the pattern for}%
\typeout{** the default language instead.}%
\else
\language=\csname l@#1\endcsname
\fi
#2}}
\providecommand{\BIBdecl}{\relax}
\BIBdecl

\bibitem{Dembo-Cover-Thomas-inequalities}
A.~Dembo, T.~M. Cover, and J.~A. Thomas, ``Information theoretic
  inequalities,'' \emph{{IEEE} Transactions on Information Theory}, vol.~37,
  pp. 1501--1518, 1991.

\bibitem{liang2008gene}
K.-C. Liang and X.~Wang, ``Gene regulatory network reconstruction using
  conditional mutual information,'' \emph{EURASIP Journal on Bioinformatics and
  Systems Biology}, vol. 2008, no.~1, p. 253894, 2008.

\bibitem{hartemink2005reverse}
A.~J. Hartemink, ``Reverse engineering gene regulatory networks,'' \emph{Nature
  biotechnology}, vol.~23, no.~5, p. 554, 2005.

\bibitem{zhang2011inferring}
X.~Zhang, X.-M. Zhao, K.~He, L.~Lu, Y.~Cao, J.~Liu, J.-K. Hao, Z.-P. Liu, and
  L.~Chen, ``Inferring gene regulatory networks from gene expression data by
  path consistency algorithm based on conditional mutual information,''
  \emph{Bioinformatics}, vol.~28, no.~1, pp. 98--104, 2011.

\bibitem{numata2008measuring}
J.~Numata, O.~Ebenh{\"o}h, and E.-W. Knapp, ``Measuring correlations in
  metabolomic networks with mutual information,'' in \emph{Genome Informatics
  2008: Genome Informatics Series Vol. 20}.\hskip 1em plus 0.5em minus
  0.4em\relax World Scientific, 2008, pp. 112--122.

\bibitem{Shannon-1948}
C.~E. Shannon, ``A mathematical theory of communication,'' \emph{Bell System
  Technical Journal}, vol.~27, pp. 379--423, 1948, reprinted in
  \cite{Shannon-and-Weaver}.

\bibitem{Cover-and-Thomas-2nd}
T.~M. Cover and J.~A. Thomas, \emph{Elements of Information Theory},
  2nd~ed.\hskip 1em plus 0.5em minus 0.4em\relax New York: John Wiley, 2006.

\bibitem{gray2011entropy}
R.~M. Gray, \emph{Entropy and Information Theory}.\hskip 1em plus 0.5em minus
  0.4em\relax Springer Science \& Business Media, 2011.

\bibitem{dembo2019probability}
A.~Dembo, ``Probability theory: Stat310/math230 apr 23, 2019,'' 2019.

\bibitem{Victor-information-bias}
J.~D. Victor, ``Asymptotic bias in information estimates and the exponential
  ({Bell}) polynomials,'' \emph{Neural Computation}, vol.~12, pp. 2797--2804,
  2000.

\bibitem{dmitriev1973functionalEst}
Y.~G. Dmitriev and F.~Tarasenko, ``On estimation of functionals of the
  probability density function and its derivatives,'' \emph{Teoriya
  veroyatnostei i ee primeneniya}, vol.~18, no.~3, pp. 662--668, 1973.

\bibitem{darbellay1999estimation}
G.~A. Darbellay and I.~Vajda, ``Estimation of the information by an adaptive
  partitioning of the observation space,'' \emph{IEEE Transactions on
  Information Theory}, vol.~45, no.~4, pp. 1315--1321, 1999.

\bibitem{kozachenko1987sample}
L.~Kozachenko and N.~N. Leonenko, ``Sample estimate of the entropy of a random
  vector,'' \emph{Problemy Peredachi Informatsii}, vol.~23, no.~2, pp. 9--16,
  1987.

\bibitem{wang2005volumes}
X.~Wang, ``Volumes of generalized unit balls,'' \emph{Mathematics Magazine},
  vol.~78, no.~5, pp. 390--395, 2005.

\bibitem{gao2018demystifying}
W.~Gao, S.~Oh, and P.~Viswanath, ``Demystifying fixed k-nearest neighbor
  information estimators,'' \emph{IEEE Transactions on Information Theory},
  2018.

\bibitem{kraskov2004estimating}
A.~Kraskov, H.~St{\"o}gbauer, and P.~Grassberger, ``Estimating mutual
  information,'' \emph{Physical review E}, vol.~69, no.~6, p. 066138, 2004.

\bibitem{frenzel2007partial}
S.~Frenzel and B.~Pompe, ``Partial mutual information for coupling analysis of
  multivariate time series,'' \emph{Physical review letters}, vol.~99, no.~20,
  p. 204101, 2007.

\bibitem{vejmelka2008inferring}
M.~Vejmelka and M.~Palu{\v{s}}, ``Inferring the directionality of coupling with
  conditional mutual information,'' \emph{Physical Review E}, vol.~77, no.~2,
  p. 026214, 2008.

\bibitem{tsimpiris2012nearest}
A.~Tsimpiris, I.~Vlachos, and D.~Kugiumtzis, ``Nearest neighbor estimate of
  conditional mutual information in feature selection,'' \emph{Expert Systems
  with Applications}, vol.~39, no.~16, pp. 12\,697--12\,708, 2012.

\bibitem{runge2018conditional}
J.~Runge, ``Conditional independence testing based on a nearest-neighbor
  estimator of conditional mutual information,'' in \emph{International
  Conference on Artificial Intelligence and Statistics}, 2018, pp. 938--947.

\bibitem{rahimzamani2017cmi}
A.~Rahimzamani and S.~Kannan, ``Potential conditional mutual information:
  Estimators and properties,'' in \emph{2017 55th Annual Allerton Conference on
  Communication, Control, and Computing (Allerton)}.\hskip 1em plus 0.5em minus
  0.4em\relax IEEE, 2017, pp. 1228--1235.

\bibitem{gao2017estimating}
W.~Gao, S.~Kannan, S.~Oh, and P.~Viswanath, ``Estimating mutual information for
  discrete-continuous mixtures,'' in \emph{Advances in Neural Information
  Processing Systems}, 2017, pp. 5988--5999.

\bibitem{rahimzamani2018mixedcmi}
A.~Rahimzamani, H.~Asnani, P.~Viswanath, and S.~Kannan, ``Estimators for
  multivariate information measures in general probability spaces,'' in
  \emph{Advances in Neural Information Processing Systems}, 2018, pp.
  8664--8675.

\bibitem{casella2002}
G.~Casella and R.~L. Berger, \emph{Statistical inference}.\hskip 1em plus 0.5em
  minus 0.4em\relax Duxbury Pacific Grove, CA, 2002, vol.~2.

\bibitem{boucheron2013concentration}
S.~Boucheron, G.~Lugosi, and P.~Massart, \emph{Concentration inequalities: A
  nonasymptotic theory of independence}.\hskip 1em plus 0.5em minus 0.4em\relax
  Oxford university press, 2013.

\bibitem{gyorfi2006distributionfree}
L.~Gy{\"o}rfi, M.~Kohler, A.~Krzyzak, and H.~Walk, \emph{A distribution-free
  theory of nonparametric regression}.\hskip 1em plus 0.5em minus 0.4em\relax
  Springer Science \& Business Media, 2006.

\bibitem{rudin1987real}
W.~Rudin, ``Real and complex analysis,'' 1987.

\bibitem{Shannon-and-Weaver}
C.~E. Shannon and W.~Weaver, \emph{The Mathematical Theory of
  Communication}.\hskip 1em plus 0.5em minus 0.4em\relax Urbana, Illinois:
  University of Illinois Press, 1963.

\end{thebibliography}





\end{document}